\documentclass[11pt]{amsart}

\usepackage{amsmath, amsthm, amssymb}
\usepackage{enumerate}

\newtheorem{theorem}{Theorem}
\newtheorem{lemma}{Lemma}
\newtheorem{proposition}{Proposition}
\newtheorem{corollary}{Corollary}

\theoremstyle{definition}

\theoremstyle{remark}
\newtheorem{remark}{Remark}

\numberwithin{equation}{section}

\newcommand{\ds}{\displaystyle}
\newcommand{\dv}{\text{div}}

\begin{document}
\title[A characterization of fast decaying solutions]{A characterization of fast decaying solutions for quasilinear and Wolff type systems with singular coefficients}


\author{John Villavert}
\address{Department of Mathematics, University of Oklahoma, Norman, OK 73019, USA}
\curraddr{}
\email{villavert@math.ou.edu,john.villavert@gmail.com}
\thanks{}

\subjclass[2010]{Primary 35B40, 45G05, 45M05; Secondary: 35J92} 

\keywords{Fast decay rates; quasilinear system; weighted Hardy--Littlewood--Sobolev inequality; Wolff potential}

\date{}

\dedicatory{}



\begin{abstract}
This paper examines the decay properties of positive solutions for a family of fully nonlinear systems of integral equations containing Wolf potentials and Hardy weights. This class of systems includes examples which are closely related to the Euler--Lagrange equations for several classical inequalities such as the Hardy--Sobolev and Hardy--Littlewood--Sobolev inequalities. In particular, a complete characterization of the fast decaying ground states in terms of their integrability is provided in that bounded and fast decaying solutions are shown to be equivalent to the integrable solutions. In generating this characterization, additional properties for the integrable solutions, such as their boundedness and optimal integrability, are also established. Furthermore, analogous decay properties for systems of quasilinear equations of the weighted Lane--Emden type are also obtained.
\end{abstract}

\maketitle

\section{Introduction}

In this paper, we examine the decay properties of positive solutions at infinity for the following class of integral systems with variable coefficients involving the Wolff potentials and Hardy weights,
\begin{equation}\label{Wolff}
  \left\{\begin{array}{l}
    u(x) = c_{1}(x)W_{\beta,\gamma}(|y|^{\sigma_1}v^q)(x), \\
    v(x) = c_{2}(x)W_{\beta,\gamma}(|y|^{\sigma_2}u^p)(x).
  \end{array}
\right.
\end{equation}
Here, the Wolff potential of a function $f$ in $L^{1}_{loc}(\mathbb{R}^n)$ is defined by
\begin{equation*}
W_{\beta,\gamma}(f)(x) = \int_{0}^{\infty} \Big( \frac{\int_{B_{t}(x)} f(y)\,dy}{t^{n-\beta\gamma}} \Big)^{\frac{1}{\gamma - 1}} \,\frac{dt}{t},
\end{equation*}
where $n \geq 3$, $\gamma > 1$, $\beta > 0$ with $\beta\gamma < n$, and $B_{t}(x) \subset \mathbb{R}^n$ denotes the ball of radius $t$ centered at $x$. Additionally, we take $p,q > 1$, $\sigma_i \leq 0$ and assume the coefficients $c_{1}(x)$ and $c_{2}(x)$ are double bounded functions i.e. there exists a positive constant $C>0$ such that $C^{-1} \leq c_{i}(x) \leq C$ for all $x \in \mathbb{R}^n$. The goal of this paper is to determine the sufficient and necessary conditions that completely describe the fast decaying ground states of system \eqref{Wolff}. One motivation for studying the decay properties of solutions for these systems stems from the fact that it is an important ingredient in the classification of solutions and in establishing Liouville type theorems. Another motivation originates from the study of the asymptotic behavior of solutions for elliptic equations. Namely, as we shall discuss below in greater detail, the integral systems we consider are natural generalizations of many elliptic equations, including the weighted equation
$$-\Delta u(x) = |x|^{\sigma}u(x)^p, \, x \in \mathbb{R}^n, \, \sigma > -2.$$
If $p > \frac{n+\sigma}{n-2}$ so that $n-2 > \frac{2+\sigma}{p-1}$, the authors in \cite{Lei13a,YLi92,LiNi88} established that ground states for this equation vanish at infinity with either the slow rate or the fast rate, respectively:
$$u(x) \simeq |x|^{-\frac{2+\sigma}{p-1}} \,\text{ or }\, u(x) \simeq |x|^{-(n-2)}.$$ 
Here, the notation $f(x) \simeq g(x)$ means there exist positive constants $c_1$ and $c_2$ such that $$c_{1}g(x) \leq f(x) \leq c_{2}g(x) \,\text{ as }\, |x|\longrightarrow\infty. $$
Hence, in a sense,  our results extend this example considerably since the Wolff potential has applications to many nonlinear problems and system \eqref{Wolff} includes several well-known cases. For instance, if $\beta = \alpha/2$ and $\gamma = 2$, the Wolff potential $W_{\beta,\gamma}(\cdot)$ becomes the Riesz potential $I_{\alpha}(\cdot )$ modulo a constant since 
\begin{align*}
W_{\frac{\alpha}{2},2}(f)(x) = {} & \int_{0}^{\infty} \frac{\int_{B_{t}(x)} f(y)\,dy}{t^{n- \alpha}}  \,\frac{dt}{t} 
= \int_{\mathbb{R}^n} f(y) \Big( \int_{|x-y|}^{\infty} t^{\alpha - n}\,\frac{dt}{t} \Big) \,dy \\
= {} & \frac{1}{(n-\alpha)}\int_{\mathbb{R}^n} \frac{f(y)\,dy}{|x-y|^{n-\alpha}} \doteq C(n,\alpha) I_{\alpha}(f)(x).
\end{align*}
Therefore, we can recover from \eqref{Wolff} the weighted version of the Hardy--Littlewood--Sobolev (HLS) system of integral equations:
\begin{equation}\label{wHLS}
  \left\{\begin{array}{l}
    u(x) = \ds \int_{\mathbb{R}^n} \frac{|y|^{\sigma_1}v(y)^q}{|x-y|^{n-\alpha}} \,dy, \medskip \\
    v(x) = \ds \int_{\mathbb{R}^n} \frac{|y|^{\sigma_2}u(y)^p}{|x-y|^{n-\alpha}} \,dy.
  \end{array}
\right.
\end{equation}
If $\sigma_i = 0$ with the critical condition
\begin{equation*}
\frac{1}{1 + q} + \frac{1}{1 + p} = \frac{n-\alpha}{n},
\end{equation*}
system \eqref{wHLS} comprises of the Euler--Lagrange equations for a functional associated with the sharp Hardy--Littlewood--Sobolev inequality. In the special case where $p=q=\frac{n+\alpha}{n-\alpha}$, Lieb classified all the maximizers for this functional, thereby obtaining the best constant in the HLS inequality. He then posed the classification of all the critical points of the functional, or the solutions of the HLS system, as an open problem \cite{Lieb83}. This conjecture on the classification of solutions was later addressed by Chen, Li and Ou \cite{CLO06} by introducing a version of the method of moving planes for integral equations (see also \cite{YYLi04} for an alternative proof via the method of moving spheres). If $\sigma_i = \sigma$, $p=q$ and $u \equiv v$, system \eqref{wHLS} reduces to the single integral equation
\begin{equation*}
u(x) = \int_{\mathbb{R}^n} \frac{|y|^{\sigma}u(y)^p}{|x-y|^{n-\alpha}} \,dy,
\end{equation*}
which is the Euler--Lagrange equation for the classical Hardy--Sobolev inequality when $\alpha = 2$ and $p = \frac{n+ 2 +2\sigma}{n-2}$. We refer the reader to \cite{Lei13a,LuZhu11} for further discussions and results on the asymptotic, symmetry, and regularity properties of solutions for this integral equation.

Interestingly, the Wolff type integral equations are also closely related to some well-known systems of differential equations. For example, if $\alpha = 2k$ is an even integer, system \eqref{wHLS} is equivalent, under the appropriate conditions (see \cite{CL13,Villavert:14b}), to the poly-harmonic system 
\begin{equation}\label{wPDE}
  \left\{\begin{array}{cl}
    (-\Delta)^{k}u(x) = |x|^{\sigma_1}v(x)^q, & x \in \mathbb{R}^n \backslash \{0\}, \\
    (-\Delta)^{k}v(x) = |x|^{\sigma_2}u(x)^p, & x \in \mathbb{R}^n \backslash \{0\}.
  \end{array}
\right.
\end{equation}
Recently, the study on the criteria governing the existence and non-existence of solutions for both differential and integral versions of the HLS type systems has received much attention, especially since Liouville type theorems are crucial in deriving a priori estimates and singularity and regularity properties of solutions for a class of nonlinear elliptic problems \cite{GS81apriori,PQS07}. More precisely, it is conjectured that either system \eqref{wHLS} or \eqref{wPDE} admits no positive solution in the subcritical case $\frac{n+\sigma_1}{1+q} + \frac{n+\sigma_2}{1+p} > n-\alpha$ (see \cite{Caristi2008,Fazly14,LGZ06,Mitidieri96,Phan12,Villavert:14c} for partial results). In the case where $\alpha=2$ and $\sigma_i = 0$, this is often referred to as the Lane--Emden conjecture and it too has only partial results. Namely, the result holds true for radial solutions (see \cite{Mitidieri96}) and for dimension $n\leq 4$ (see \cite{PQS07,SZ96,Souplet09}). On the other hand, the scalar analogue of this conjecture is classical and has a complete solution (see \cite{CGS89,CL91,YYLiZhu95}). Conversely, we refer the reader to \cite{Villavert:14b} (see also \cite{ChengLi14,Li13}) for existence results to system \eqref{wPDE} in the non-subcritical case $$\frac{n+\sigma_1}{1+q} + \frac{n+\sigma_2}{1+p} \leq n-2k.$$ 
The Wolff type integral systems are also closely related to many other notable differential equations. For instance, if $\beta = 1$, the equation 
$$ u(x) = W_{1,\gamma}(|y|^{\sigma}u^q)(x)$$
corresponds to the $\gamma$-Laplace equation
$$-\dv (|\nabla u|^{\gamma-2}\nabla u) = |x|^{\sigma}u(x)^q.$$
More generally, if $\beta = \frac{2k}{k+1}$ and $\gamma = k+1$, then the integral equation
$$ u(x) = W_{\frac{2k}{k+1},k+1}(|y|^{\sigma}u^q)(x)$$
corresponds to the $k$-Hessian equation $$F_{k}[-u] = |x|^{\sigma}u(x)^q, \, \text{ for }\, k=1,2,\ldots,n,$$
where
$$F_{k}[u] = S_{k}(\lambda(D^2 u)), \, \lambda(D^2 u) = (\lambda_1,\lambda_2,\ldots,\lambda_n),$$
and the $\lambda_i$'s are the eigenvalues of the Hessian matrix $D^{2}u$ and $S_{k}(\cdot)$ is the $k^{th}$ symmetric function
$$ S_{k}(\lambda) = \ds\sum_{1\leq i_i < \ldots < i_k \leq n} \lambda_{i_1}\lambda_{i_2}\cdots \lambda_{i_k}.$$
Notice that when $k=1$ and $k=n$, we recover the familiar second-order elliptic operators:
$$ F_{1}[u] = \Delta u \,\text{ and }\, F_{n}[u] =  \det(D^{2}u).$$
Let us also discuss previous works concerning integral systems involving the Wolff potentials. In particular, the qualitative properties of solutions for the unweighted version of system \eqref{Wolff} and its special cases have been studied by a number of authors. For instance, the authors in \cite{MCL11} studied the integrability and regularity of solutions, and the authors in \cite{Lei13b}-\cite{LLM12} and \cite{SunLei2012} examined the asymptotic and symmetry properties of solutions. Similar qualitative results were obtained in \cite{ChenLu14} for a more specific weighted integral system of Wolff type under different and often times stronger assumptions compared to those in this paper. For more on the properties of the Wolff potentials and other related problems, we refer the reader to \cite{KM92,KM94,NguyenVeron14,PhucVerbitsky08}. 

\section{Some preliminaries and the main results}
Throughout this paper we shall further assume that $\gamma \in (1,2]$ and $\sigma_i \in (-\beta\gamma,0]$. We shall also take the coefficients $c_{1}(x)$ and $c_{2}(x)$ of \eqref{Wolff} to be double bounded. In characterizing the fast decaying ground states for the integral systems, we shall consider the integrable solutions. Namely, we say a positive solution $(u,v)$ of system \eqref{Wolff} is an {\bf integrable solution} if $(u,v) \in L^{r_0}(\mathbb{R}^n) \times L^{s_0}(\mathbb{R}^n)$ with $r_0 = \frac{n}{q_0}$ and $s_0 = \frac{n}{p_0}$ where
\begin{align*}
q_0 \doteq {} & \frac{\beta\gamma(\gamma - 1 + q) + (\gamma - 1)\sigma_1 + \sigma_2 q}{pq - (\gamma-1)^2} \,\text{ and } \\ p_0 \doteq {} & \frac{\beta\gamma(\gamma - 1 + p) + (\gamma - 1)\sigma_2 + \sigma_1 p}{pq - (\gamma-1)^2}.
\end{align*}
In view of the Lane--Emden and HLS conjectures and the related non-existence results cited above, we always assume hereafter the non-subcritical condition $$q_0 + p_0 \leq \frac{n-\beta\gamma}{\gamma-1},$$ or equivalently
\begin{equation}\label{non-subcritical} 
\frac{n+\sigma_1}{\gamma -1 +q} + \frac{n+\sigma_2}{\gamma -1 +p} \leq \frac{n-\beta\gamma}{\gamma-1}.
\end{equation}
Then, our main result states that integrable solutions are exactly those ground states which decay with the fast rates. 
\begin{theorem}\label{fast theorem}
Let $q\geq p$ and $\sigma_1 \leq \sigma_2 \leq 0$ and let $u,v$ be a positive solution of the integral system \eqref{Wolff} satisfying \eqref{non-subcritical}. Then $u,v$ are integrable solutions if and only if $u,v$ are bounded and decay with the fast rates as $|x|\longrightarrow \infty$:
$$u(x) \simeq |x|^{-\frac{n-\beta\gamma}{\gamma-1}}$$
and
\begin{equation*}
\left\{\begin{array}{ll}
v(x) \simeq |x|^{-\frac{n-\beta\gamma}{\gamma-1}}, 							                     & \text{ if }\, p(\frac{n-\beta\gamma}{\gamma-1}) - \sigma_2 > n; \\
v(x) \simeq |x|^{-\frac{n-\beta\gamma}{\gamma-1}}(\ln |x|)^{\frac{1}{\gamma-1}},	                 & \text{ if }\, p(\frac{n-\beta\gamma}{\gamma-1}) - \sigma_2 = n; \\
v(x) \simeq |x|^{- \frac{p(\frac{n-\beta\gamma}{\gamma-1}) - (\beta\gamma + \sigma_2)}{\gamma-1}}, & \text{ if }\, p(\frac{n-\beta\gamma}{\gamma-1}) - \sigma_2 < n.
\end{array}
\right.
\end{equation*}
\end{theorem}
This theorem essentially contains the decay properties of solutions for the weighted HLS type integral system, which can also be found in \cite{Villavert:14d}.

\begin{corollary}\label{cor1}
Let $q\geq p$, $\sigma_1 \leq \sigma_2$ and let $u,v$ be a positive solution of system \eqref{wHLS} satisfying the non-subcritical condition
\begin{equation*}
\frac{n+\sigma_1}{1+q} + \frac{n+\sigma_2}{1+p} \leq n-\alpha.
\end{equation*}
Then $u,v$ are integrable solutions if and only if $u,v$ are bounded and decay with the fast rates as $|x|\longrightarrow \infty$:
$$u(x) \simeq |x|^{-(n-\alpha)}$$
and
\begin{equation*}
\left\{\begin{array}{ll}
v(x) \simeq |x|^{-(n-\alpha)},                           & \text{ if }\, p(n-\alpha) - \sigma_2 > n; \\
v(x) \simeq |x|^{-(n-\alpha)}\ln |x|,                    & \text{ if }\, p(n-\alpha) - \sigma_2 = n; \\
v(x) \simeq |x|^{-(p(n-\alpha) - (\alpha + \sigma_2))},  & \text{ if }\, p(n-\alpha) - \sigma_2 < n.
\end{array}
\right.
\end{equation*}
\end{corollary}

\begin{remark}
The assumptions $q \geq p$ and $\sigma_1 \leq \sigma_2$ are due to the inhomogeneity of the systems when $q \neq p$ and $\sigma_1 \neq \sigma_2$ and this illustrates a difficulty we encounter, which does not arise in the scalar case, when examining the systems. However, these assumptions are not so essential in the following sense. Indeed, we can interchange these parameters and the results of Theorem \ref{fast theorem} remain valid provided we interchange the parameters along with $u$ and $v$ accordingly in the statement of the theorem.
\end{remark}

\begin{remark}
Consider the unweighted case where $\sigma_i = 0$. 
\begin{enumerate}[(i)]
\item In \cite{MCL11} and \cite{SunLei2012}, the authors considered instead the ``finite-energy" solutions i.e. $(u,v) \in L^{p+\gamma-1}(\mathbb{R}^n)\times L^{q+\gamma -1}(\mathbb{R}^n)$ for the unweighted system \eqref{Wolff} under the critical case 
$$\frac{1}{\gamma -1 + q} + \frac{1}{\gamma -1 + p} = \frac{n-\beta\gamma}{n(\gamma-1)}.$$ 
From Theorem 1 in \cite{MCL11}, we may deduce that finite-energy solutions are integrable solutions. Conversely, in the next section we show that integrable solutions are indeed finite-energy solutions even under the weaker condition \eqref{non-subcritical}.
\item In the critical case, the particular rate for $v(x)$,
$$\frac{p(\frac{n-\beta\gamma}{\gamma-1}) - (\beta\gamma + \sigma_2)}{\gamma-1},$$ 
is equal to 
$\frac{n}{\gamma-1}\frac{\gamma-1+p}{\gamma-1+q}$
and so our main theorem coincides with the asymptotic results of \cite{SunLei2012} for the unweighted system.
\end{enumerate}
If $\sigma_i \neq 0$, system \eqref{wHLS} differs from the well-known doubly weighted HLS system in terms of the asymptotic properties of their solutions. Namely, the fast decay rates of solutions for \eqref{wHLS}, as indicated by Corollary \ref{cor1}, are different from the doubly weighted HLS system (cf. \cite{LLM12}). 
\end{remark}

As a consequence of Theorem \ref{fast theorem}, we can also establish a corresponding result for quasilinear systems. Consider the system
\begin{equation}\label{quasilinear}
  \left\{\begin{array}{c}
    -\dv\,\mathcal{A}(x,\nabla u) = c_{1}(x)|x|^{\sigma_1}v(x)^q,  \medskip \\
    -\dv\,\mathcal{A}(x,\nabla v) = c_{2}(x)|x|^{\sigma_2}u(x)^p,
  \end{array}
\right.
\end{equation}
where $c_{1}(x)$ and $c_{2}(x)$ are double bounded and the map $\mathcal{A}:\mathbb{R}^n \times \mathbb{R}^n \mapsto \mathbb{R}^n$ satisfies the following properties. The mapping $x \mapsto \mathcal{A}(x,\xi)$ is measurable for all $\xi \in \mathbb{R}^n;$ the mapping $\xi \mapsto \mathcal{A}(x,\xi)$ is continuous for a.e. $x \in \mathbb{R}^n$; for some positive constants $k_{1}\leq k_{2}$ there hold for all $\xi\in\mathbb{R}^n$ and a.e. $x \in \mathbb{R}^n$,
\begin{enumerate}[(a)]
\item $\mathcal{A}(x,\xi)\cdot\xi \geq k_{1}|\xi|^{\gamma}$,
\item $|\mathcal{A}(x,\xi)|\leq k_{2}|\xi|^{\gamma-1}$,
\item $(\mathcal{A}(x,\xi) - \mathcal{A}(x,\xi'))\cdot(\xi - \xi') > 0$ whenever $\xi \neq \xi'$,
\item $\mathcal{A}(x,\lambda\xi) = \lambda|\lambda|^{\gamma-2}\mathcal{A}(x,\xi)$ for all $\lambda \neq 0$.
\end{enumerate}

\begin{remark}
In the simple case where $\mathcal{A}(x,\xi) \doteq |\xi|^{\gamma-2}\xi$, $\dv\,\mathcal{A}(x,\nabla u)$ becomes the usual $\gamma$-Laplace operator $\dv(|\nabla u|^{\gamma-2}\nabla u)$. Moreover, positive solutions of \eqref{quasilinear} are to be understood in the usual weak sense i.e. $u,v \in W^{1,\gamma}_{loc}(\mathbb{R}^n) \cap C(\mathbb{R}^n)$ satisfying the system in the distribution sense.
\end{remark}

\begin{corollary}\label{cor2}
Let $(u,v)$ be a positive solution of system \eqref{quasilinear} satisfying the associated non-subcritical condition. Then $u,v$ are integrable solutions if and only if $u,v$ are bounded and decay with the fast rates as $|x|\longrightarrow \infty$:
$$u(x) \simeq |x|^{-\frac{n-\gamma}{\gamma-1}}$$
and
\begin{equation*}
\left\{\begin{array}{ll}
v(x) \simeq |x|^{-\frac{n-\gamma}{\gamma-1}},                                            & \text{ if }\, p(\frac{n-\gamma}{\gamma-1}) - \sigma_2 > n; \\
v(x) \simeq |x|^{-\frac{n-\gamma}{\gamma-1}}(\ln |x|)^{\frac{1}{\gamma-1}},              & \text{ if }\, p(\frac{n-\gamma}{\gamma-1}) - \sigma_2 = n; \\
v(x) \simeq |x|^{- \frac{p(\frac{n-\gamma}{\gamma-1}) - (\gamma + \sigma_2)}{\gamma-1}}, & \text{ if }\, p(\frac{n-\gamma}{\gamma-1}) - \sigma_2 < n.
\end{array}
\right.
\end{equation*}

\end{corollary}

Let us now recall several basic estimates for both the Riesz and Wolff potentials which we often invoke throughout this paper (see \cite{JL06,MCL11}). 
\begin{lemma}\label{wHLS Wolff inequality}
Let $p,q > 1$.
\begin{enumerate}[(1)]
\item (weighted HLS type inequality) Let $\alpha \in (0,n)$, and $\sigma \in (-\alpha, 0]$. Then there exists some positive constant $C = C(n,p,\alpha,\sigma)$ such that
$$ \|I_{\alpha}(|y|^{\sigma}f)\|_{q} \leq C\|f\|_{p} \,\text{ for all }\, f \in L^{p}(\mathbb{R}^n),$$
where $ \frac{1}{p} - \frac{1}{q} = \frac{\alpha + \sigma}{n}$ and $q > \frac{n}{n-\alpha}$.

\item Let $\beta > 0$ , $\gamma > 1$, and $\beta\gamma < n$. Then there exists some positive constant $C$ such that
$$ \|W_{\beta,\gamma}(f)\|_{q} \leq C\|f\|^{\frac{1}{\gamma-1}}_{p} \,\text{ for all }\, f \in L^{p}(\mathbb{R}^n),$$
where $\frac{1}{p} - \frac{\gamma-1}{q} = \frac{\beta\gamma}{n}$ and $q > \gamma -1$.
\end{enumerate}
\end{lemma}
Moreover, we have a comparison principle between the $L^p$ norms of the Riesz and Wolff potentials (see Proposition 5.1 in \cite{PhucVerbitsky08}).  
\begin{lemma}[Wolff's inequality] \label{Wolff's inequality}
Let $p>1$, $\beta>0$, $\gamma > 1$ and $\beta\gamma < n$.
Then there exist positive constants  $C_1$ and $C_2$ such that
$$C_{1}\|W_{\beta,\gamma}(f)\|_{p} \leq \|I_{\beta\gamma}(f)\|_{\frac{p}{\gamma-1}}^{\frac{1}{\gamma-1}} \leq C_{2}\|W_{\beta,\gamma}(f)\|_{p},$$
\end{lemma}

The remaining parts of this paper are organized in the following way. In Section \ref{section properties}, we establish several important qualitative properties of integrable solutions which are essential in our proof of Theorem \ref{fast theorem}, including an optimal integrability result. In the same section, a boundedness property is given in Theorem \ref{boundedness property} which is another key ingredient in establishing the fast decay rates of integrable solutions. However, we delay its proof until Section \ref{proof of theorem 3} in order to better illustrate the main ideas in the proof of Theorem \ref{fast theorem}. Then, Section \ref{section fast decay} and Section \ref{proof of corollary} contains the proof of Theorem \ref{fast theorem} and Corollary \ref{cor2}, respectively. Moreover, we should mention that some of our methods below are inspired by those from \cite{LL12} and \cite{SunLei2012}.

\section{Properties of integrable solutions}\label{section properties}
First, we establish an optimal integrability result for integrable solutions and show that they are indeed ground states.

\begin{theorem}\label{integrability theorem}
Suppose $q\geq p$ and $\sigma_1 \leq \sigma_2$. If $u,v$ are positive integrable solutions of \eqref{Wolff}, then $(u,v) \in L^{r}(\mathbb{R}^n) \times L^{s}(\mathbb{R}^n)$ where
\begin{equation}\label{interval}
\frac{n(\gamma-1)}{n-\beta\gamma} < r \leq \infty \,\text{ and }\, \max \Big\lbrace \frac{n(\gamma-1)}{n-\beta\gamma}, \frac{n(\gamma-1)} {p(\frac{n-\beta\gamma}{\gamma-1}) - (\beta\gamma +\sigma_2) }\Big\rbrace < s \leq \infty. 
\end{equation}
Furthermore, $u,v \longrightarrow 0$ as $|x| \longrightarrow \infty$.
\end{theorem}

\begin{remark}
The intervals in \eqref{interval} are indeed optimal. Namely, there hold $\|u\|_{r} = \|v\|_{s} = \infty$ at the endpoints 
$$ r = \frac{n(\gamma-1)}{n-\beta\gamma} \,\text{ and }\, s = \max \Big\lbrace \frac{n(\gamma-1)}{n-\beta\gamma}, \frac{n(\gamma-1)} {p(\frac{n-\beta\gamma}{\gamma-1}) - (\beta\gamma +\sigma_1) }\Big\rbrace.$$
To see this, notice that
$$u(x) \geq c\int_{|x|}^{2|x|} \Big( \frac{\int_{B_{t}(x)} |y|^{\sigma_1} v(y)^q \,dy}{t^{n-\beta\gamma}} \Big)^{\frac{1}{\gamma-1}} \,\frac{dt}{t} 
\geq  c\int_{|x|}^{2|x|} t^{-\frac{n-\beta\gamma}{\gamma-1} } \, \frac{dt}{t} \geq c|x|^{-\frac{n-\beta\gamma}{\gamma-1}}$$
and similarly 
$$v(x) \geq c|x|^{-\frac{n-\beta\gamma}{\gamma-1}}.$$
Therefore, the first estimate implies that for $u$ to belong to $L^{r}(\mathbb{R}^n)$, then $(n-\beta\gamma)(\gamma-1)^{-1}r > n$ or $r > \frac{n(\gamma-1)}{n-\beta\gamma}$. The lower bound of $v(x)$ implies that
\begin{align*}
u(x) \geq {} & c\int_{0}^{|x|} \Big( \frac{\int_{B_{t}(x)} |y|^{\sigma_1} v(y)^q \,dy}{t^{n-\beta\gamma}} \Big)^{\frac{1}{\gamma-1}}  \,\frac{dt}{t} \\
\geq {} & c\int_{0}^{|x|} \Big( \frac{\int_{B_{t}(x)} |y|^{-q\frac{n-\beta\gamma}{\gamma-1} + \sigma_1} \,dy}{t^{n-\beta\gamma}} \Big)^{\frac{1}{\gamma-1}} \,\frac{dt}{t} \\
\geq {} & c|x|^{-\frac{q(n-\beta\gamma)}{(\gamma-1)^2} + \frac{\beta\gamma + \sigma_1}{\gamma -1}}.
\end{align*}
Then, if $u$ belongs to $L^{r}(\mathbb{R}^n)$, then it necessarily holds that 
$$ r > \frac{n(\gamma-1)}{q(\frac{n-\beta\gamma}{\gamma-1}) - (\beta\gamma + \sigma_1)}. $$
Thus, in view of $q \geq p$ and $\sigma_1 \leq \sigma_2$, the necessary condition for $u$ to belong to $L^{r}(\mathbb{R}^n)$ is
$$ r > \max \Big\lbrace \frac{n(\gamma-1)}{n-\beta\gamma}, \frac{n(\gamma-1)}{q(\frac{n-\beta\gamma}{\gamma-1}) - (\beta\gamma + \sigma_1)}    \Big\rbrace = \frac{n(\gamma-1)}{n-\beta\gamma}.$$
Likewise, we can show the necessary condition for $v$ to belong to $L^{s}(\mathbb{R}^n)$ is
$$ s > \max \Big\lbrace \frac{n(\gamma-1)}{n-\beta\gamma}, \frac{n(\gamma-1)}{p(\frac{n-\beta\gamma}{\gamma-1}) - (\beta\gamma + \sigma_2)} \Big\rbrace.$$
\end{remark}

\begin{proof}[Proof of Theorem \ref{integrability theorem}]
Due to the double bounded property, we may assume without loss of generality, that $c_{1}(x), c_{2}(x) \equiv 1$. Set $a = 1/r_0$, $b = 1/s_0$ and let 
$$I \doteq (a-b, \frac{n-\beta\gamma}{n(\gamma-1)}) \times (0,\frac{n-\beta\gamma}{n(\gamma-1)} - a + b). $$ 
Note that $a - b \geq 0$ since $q \geq p$ and $\sigma_1 \leq \sigma_2 \leq 0$.
\medskip

\noindent{\it Step 1:} We first establish the integrability of solutions in the smaller interval $I$, then we extend to the larger interval as stated in the theorem. Choose any pair of positive real numbers $r$ and $s$ such that $(1/r,1/s) \in I$ and
\begin{equation}\label{rs constraint}
\frac{1}{r} - \frac{1}{s} = \frac{1}{r_0} - \frac{1}{s_0} = a - b.
\end{equation}
It follows that $$ \frac{1}{r} - \frac{2-\gamma}{r_0} + \frac{\beta\gamma + \sigma_1}{n} = \frac{q-1}{s_0} + \frac{1}{s} \,\text{ and }\, \frac{1}{s} - \frac{2-\gamma}{s_0} + \frac{\beta\gamma + \sigma_2}{n} = \frac{p-1}{r_0} + \frac{1}{r}.$$
For a fixed real number $A>0$ and some given function $w(x)$, we associate to it the function $w_{A}(x)$ defined
\begin{equation*}
w_{A}(x) =  
\left\{\begin{array}{cc}
    w(x) & \mbox{ if } w(x) > A \mbox{ or } |x| > A, \\
    0    & \mbox{ otherwise. }
  \end{array}
\right.
\end{equation*}
Define the integral operator $T(f,g) = (T_{1}g,T_{2}f)$ where
$$ T_{1}g(x) = \int_{0}^{\infty} \Big( \frac{\int_{B_{t}(x)} |y|^{\sigma_1}v(y)^q \,dy}{t^{n-\beta\gamma}} \Big)^{\frac{2-\gamma}{\gamma-1}} \Big( \frac{\int_{B_{t}(x)} |y|^{\sigma_1}v_{A}(y)^{q-1}g(y) \,dy}{t^{n-\beta\gamma}} \Big) \,\frac{dt}{t},$$
and
$$ T_{2}f(x) = \int_{0}^{\infty} \Big( \frac{\int_{B_{t}(x)} |y|^{\sigma_2}u(y)^p \,dy}{t^{n-\beta\gamma}} \Big)^{\frac{2-\gamma}{\gamma-1}} \Big( \frac{\int_{B_{t}(x)} |y|^{\sigma_2}u_{A}(y)^{p-1}f(y) \,dy}{t^{n-\beta\gamma}} \Big) \,\frac{dt}{t}.$$
Moreover, let 
$$ F = \int_{0}^{\infty} \Big( \frac{\int_{B_{t}(x)} |y|^{\sigma_1}v(y)^q \,dy}{t^{n-\beta\gamma}} \Big)^{\frac{2-\gamma}{\gamma-1}} \Big( \frac{\int_{B_{t}(x)} |y|^{\sigma_1}(v(y)-v_{A}(y))^{q} \,dy}{t^{n-\beta\gamma}} \Big) \,\frac{dt}{t},$$
and
$$ G = \int_{0}^{\infty} \Big( \frac{\int_{B_{t}(x)} |y|^{\sigma_2}u(y)^p \,dy}{t^{n-\beta\gamma}} \Big)^{\frac{2-\gamma}{\gamma-1}} \Big( \frac{\int_{B_{t}(x)} |y|^{\sigma_2} (u(y) - u_{A}(y))^{p} \,dy}{t^{n-\beta\gamma}} \Big) \,\frac{dt}{t}.$$
Clearly, a positive solution $u,v$ of system \eqref{Wolff} satisfies 
\begin{equation}\label{lift eq}
(u,v) = T(u,v) + (F,G).
\end{equation}
By H\"{o}lder's inequality,
\begin{align*}
|T_{1}g(x)| \leq {} & u(x)^{2-\gamma}\Bigg\lbrace \int_{0}^{\infty} \Big( \frac{\int_{B_{t}(x)} |y|^{\sigma_1}v_{A}(y)^{q-1}g(y)\,dy}{t^{n-\beta\gamma}} \Big)^{\frac{1}{\gamma-1}} \,\frac{dt}{t} \Bigg\rbrace^{\gamma-1} \\
\doteq {} & u(x)^{2-\gamma}\Big\lbrace T_{1}^{0}g(x)\Big \rbrace^{\gamma-1}.
\end{align*}
Therefore,
\begin{equation*}
\|T_{1}g\|_{r} \leq \|u\|_{r_0}^{2-\gamma}\|T_{1}^{0}g\|_{\overline{r}}^{\gamma-1},
\end{equation*}
where $\frac{1}{r} = \frac{2-\gamma}{r_0} + \frac{\gamma-1}{\overline{r}}$. Then, by applying the Wolff's inequality of Lemma \ref{Wolff's inequality} followed by the weighted HLS inequality, we get
\begin{align*}
\|T_{1}^{0}g\|_{\overline{r}}^{\gamma-1} \leq {} & \| W_{\beta,\gamma}(|y|^{\sigma_1}v_{A}^{q-1}g)\|_{\overline{r}}^{\gamma-1} \leq C\| I_{\beta\gamma}(|y|^{\sigma_1}v_{A}^{q-1}g)\|_{\frac{\overline{r}}{\gamma -1}} \\
\leq {} & C\|v_{A}^{q-1}g\|_{\frac{n\overline{r}}{n(\gamma-1)+ \overline{r}(\beta\gamma + \sigma_1)}}.
\end{align*}
Noticing that 
$$ \frac{n(\gamma-1)+ \overline{r}(\beta\gamma +\sigma_1)}{n\overline{r}} = \frac{\gamma-1}{\overline{r}}  + \frac{\beta\gamma + \sigma_1}{n} = \frac{1}{r} - \frac{2-\gamma}{r_0} + \frac{\beta\gamma + \sigma_1}{n}, $$
H\"{o}lder's inequality implies
\begin{equation*}
\|v_{A}^{q-1}g\|_{\frac{n\overline{r}}{n(\gamma-1)+ \overline{r}(\beta\gamma + \sigma_1)}} \leq \|v_{A}\|_{s_0}^{q-1}\|g\|_{s},
\end{equation*}
and therefore,
\begin{equation}\label{contract1}
\|T_{1}g\|_{r} \leq C_{1}\|u\|_{r_0}^{2-\gamma} \|v_{A}\|_{s_0}^{q-1}\|g\|_{s}.
\end{equation}
Likewise, there holds
\begin{equation*}
\|T_{2}f\|_{s} \leq \|v\|_{s_0}^{2-\gamma}\|T_{2}^{0}f\|_{\overline{s}}^{\gamma-1},
\end{equation*}
where $\frac{1}{s} = \frac{2-\gamma}{s_0} + \frac{\gamma-1}{\overline{s}}$ and
\begin{equation*}
T_{2}^{0}f(x) \doteq \int_{0}^{\infty} \Big( \frac{\int_{B_{t}(x)} |y|^{\sigma_2}u_{A}(y)^{p-1}f(y)\,dy}{t^{n-\beta\gamma}} \Big)^{\frac{1}{\gamma-1}} \,\frac{dt}{t}.
\end{equation*}
As before, by applying the Wolff type inequality, the weighted HLS inequality and H\"{o}lder's inequality, we arrive at the estimate
\begin{equation}\label{contract2}
\|T_{2}f\|_{s} \leq C_{2}\|v\|_{s_0}^{2-\gamma} \|u_{A}\|_{r_0}^{p-1}\|f\|_{r}.
\end{equation}
Obviously, we can choose $A$ sufficiently large so that 
$$C_{1}\|u\|_{r_0}^{2-\gamma}\|v_{A}\|_{s_0}^{q-1},\, C_{2}\|v\|_{s_0}^{2-\gamma}\|u_{A}\|_{r_0}^{p-1} \leq 1/2.$$
Hence, \eqref{contract1} and \eqref{contract2} imply that the operator $T(f,g)$ equipped with the norm
$$ \|(f_1,f_2)\|_{L^{r}(\mathbb{R}^n)\times L^{s}(\mathbb{R}^n)} \doteq \|f_1\|_{r} + \|f_2\|_{s},$$
is a contraction map from $L^{r}(\mathbb{R}^n)\times L^{s}(\mathbb{R}^n)$ to itself. Moreover, it is clear from the definition that $(F,G)$ belongs to $L^{r}(\mathbb{R}^n) \times L^{s}(\mathbb{R}^n)$. Thus, since $(u,v)$ satisfies \eqref{lift eq}, applying the regularity lifting result of Lemma 2.2 in \cite{MCL11} implies that $(u,v) \in L^{r}(\mathbb{R}^n)\times L^{s}(\mathbb{R}^n)$ for all $(1/r,1/s) \in I$.
\medskip

\noindent{\it Step 2:} We extend the interval $I$. From the first integral equation and Lemmas \ref{wHLS Wolff inequality} and \ref{Wolff's inequality}, we have
\begin{equation*}
\|u\|_{r} \leq C\|W_{\beta,\gamma}(|y|^{\sigma_1}v^q)\|_{r} \leq C\|v^q\|_{\frac{nr}{n(\gamma-1)+r(\beta\gamma+\sigma_1)}}^{\frac{1}{\gamma-1}} \leq C\|v\|_{\frac{nrq}{n(\gamma-1)+r(\beta\gamma+\sigma_1)}}^{\frac{q}{\gamma-1}}.
\end{equation*}
Since $v \in L^{s}(\mathbb{R}^n)$ for all $\frac{1}{s} \in (0,\frac{n-\beta\gamma}{n(\gamma-1)} - a + b)$, the previous estimate implies that $u\in L^{r}(\mathbb{R}^n)$ for all $\frac{1}{r} \in (0, \frac{q}{\gamma-1}\lbrace \frac{n-\beta\gamma}{n(\gamma-1)} - a + b \rbrace - \frac{\beta\gamma+\sigma_1}{n(\gamma-1)})$. From the fact that $p_0, q_0 < \frac{n-\beta\gamma}{n(\gamma-1)}$, we can easily show that
\begin{equation}\label{claim positive}
\frac{q}{\gamma-1}\Big\lbrace \frac{n-\beta\gamma}{n(\gamma-1)} - a + b \Big\rbrace - \frac{\beta\gamma+\sigma_1}{n(\gamma-1)} > a - b,
\end{equation}
and thus $u \in L^{r}(\mathbb{R}^n)$ for all $\frac{1}{r} \in (0,\frac{n-\beta\gamma}{n(\gamma-1)}).$ 

Likewise, we can apply the same arguments on the second integral equation to show that
\begin{equation*}
\|v\|_{s} \leq C\|u\|_{\frac{nsp}{n(\gamma-1)+s(\beta\gamma+\sigma_2)}}^{\frac{p}{\gamma-1}}.
\end{equation*}
Hence, since $u \in L^{r}(\mathbb{R}^n)$ for all $\frac{1}{r} \in (0,\frac{n-\beta\gamma}{n(\gamma-1)})$, we get that $v \in L^{s}(\mathbb{R}^n)$ for all 
$$ \frac{1}{s} \in \Big(0,\min \Big\lbrace \frac{n-\beta\gamma}{n(\gamma-1)},\, \frac{p(\frac{n-\beta\gamma}{\gamma-1}) - (\beta\gamma + \sigma_2)}{n(\gamma-1)} \Big\rbrace\Big).$$ 
\medskip 

\noindent{\it Step 3:} It remains to show that $u,v$ are ground states i.e. $u,v \in L^{\infty}(\mathbb{R}^n)$ and $u,v \longrightarrow 0$ as $|x|\longrightarrow \infty$.

We only show $u(x)$ is bounded and vanish at infinity, since the result for $v(x)$ follows similarly. For small $\delta \in (0,1)$, 
\begin{equation*}
u(x) \leq C\Big( \int_{0}^{\delta} + \int_{\delta}^{\infty}\Big) \Big( \frac{\int_{B_{t}(x)} |y|^{\sigma_1}v(y)^q \,dy}{t^{n-\beta\gamma}} \Big)^{\frac{1}{\gamma-1}} \,\frac{dt}{t} \doteq C(I_1 + I_2).
\end{equation*}
We shall estimate $I_1$ and $I_2$. First, we choose a suitably large $\ell>1$ with $n+ \frac{\sigma_1 \ell}{\ell-1} > 0$ so that  H\"{o}lder's inequality and Theorem \ref{integrability theorem} imply
\begin{align}\label{estimate 0}
\int_{B_{t}(x)} |y|^{\sigma_1} v(y)^q \,dy \leq {} & C\|v^q\|_{\ell} \Big( \int_{B_{t}(x)} |y|^{\frac{\sigma_1 \ell}{\ell-1}}\,dy \Big)^{1-1/\ell} \\
\leq {} & C t^{n(1-1/\ell) + \sigma_1}\|v\|_{\ell q}^{q}. \notag 
\end{align}
In the last inequality, we used estimates \eqref{estimate 1} and \eqref{estimate 2} from below. Let $t \leq |x|/2$. If $y \in B_{t}(x)$, then $|x|/2 \leq |y|$ and thus $|x-y| \leq |y|$. Therefore,
\begin{equation}\label{estimate 1}
\int_{B_{t}(x)} |y|^{\frac{\sigma_1 \ell}{\ell-1}} \,dy \leq \int_{B_{t}(x)} |x-y|^{\frac{\sigma_1 \ell}{\ell-1}}\,dy \leq Ct^{n+\frac{\sigma_1 \ell}{\ell-1}}. 
\end{equation}
On the other hand, let $t > |x|/2$. If $y \in B_{t}(x)$, then $y \in B_{t+|x|}(0)$ and 
\begin{equation}\label{estimate 2}
\int_{B_{t}(x)} |y|^{\frac{\sigma_1 \ell}{\ell-1}} \,dy \leq \int_{B_{t+|x|}(0)} |y|^{\frac{\sigma_1 \ell}{\ell-1}} \,dy \leq \int_{0}^{|x|+t} s^{n+ \frac{\sigma_1 \ell}{\ell - 1}} \,\frac{ds}{s} \leq Ct^{n+\frac{\sigma_1 \ell}{\ell-1}}.
\end{equation}
Choosing $\ell$ large enough so that we also have $\beta\gamma + \sigma_1 - n/\ell > 0$, estimate \eqref{estimate 0} implies
$$ I_{1} \leq C_{1} \int_{0}^{\delta} \Big(\frac{t^{n(1-1/\ell) + \sigma_1}}{t^{n-\beta\gamma}} \Big)^{\frac{1}{\gamma-1}} \,\frac{dt}{t} \leq C_{1}\int_{0}^{\delta} t^{\frac{\beta\gamma + \sigma_1 - n/\ell}{\gamma-1}} \,\frac{dt}{t} < \infty. $$

Choose a small $c \in (0,1)$. If $z \in B_{c}(x)$, then $B_{t}(x) \subset B_{t+c}(z)$. Thus, for $z \in B_{c}(x)$,
\begin{align*}
I_{2} \leq {} & C_{2}\int_{\delta}^{\infty} \Big( \frac{\int_{B_{t}(x)} |y|^{\sigma_1}v(y)^q \,dy}{t^{n-\beta\gamma}} \Big)^{\frac{1}{\gamma-1}} \,\frac{dt}{t} \\
\leq {} & C_{2}\int_{\delta}^{\infty} \Big( \frac{\int_{B_{t+c}(z)} |y|^{\sigma_1}v(y)^q \,dy}{(t+c)^{n-\beta\gamma}} \Big)^{\frac{1}{\gamma-1}} \Big(\frac{t+c}{t}\Big)^{\frac{n-\beta\gamma}{\gamma-1} + 1} \,\frac{d(t+c)}{t+c} \\
\leq {} & C_{2}\Big(1 + \frac{c}{t}\Big)^{\frac{n-\beta\gamma}{\gamma-1} + 1}\int_{\delta + c}^{\infty} \Big( \frac{\int_{B_{t}(z)} |y|^{\sigma_1}v(y)^q \,dy}{t^{n-\beta\gamma}} \Big)^{\frac{1}{\gamma-1}} \,\frac{dt}{t} \\
\leq {} & C_{2}u(z).
\end{align*}
Hence, combining our estimates for $I_1$ and $I_2$ give us $u(x) \leq C_{1} + C_{2}u(z)$ for all $z \in B_{c}(x)$. Integrating this inequality on the ball $B_{c}(x)$ then applying H\"{o}lder's inequality yields
$$u(x) \leq C_{1} + \frac{C_{2}}{|B_{c}(x)|}\int_{B_{c}(x)} u(z) \,dz \leq C_{1} + C_{2}|B_{c}(x)|^{-1/r_0}\|u\|_{r_0} \leq C.$$
Hence, $u \in L^{\infty}(\mathbb{R}^n)$.

For each $\epsilon > 0$, we can find a small $\delta > 0$ such that 
$$ \int_{0}^{\delta} \Big( \frac{\int_{B_{t}(x)} |y|^{\sigma_1}v(y)^q \,dy}{t^{n-\beta\gamma}} \Big)^{\frac{1}{\gamma-1}} \,\frac{dt}{t} \leq C\|v\|_{\infty}^{\frac{q}{\gamma-1}} \int_{0}^{\delta} t^{\frac{\beta\gamma+\sigma_1}{\gamma-1}}\,\frac{dt}{t} < \epsilon. $$
Using similar arguments we used in estimating $I_2$, we calculate
$$ \int_{\delta}^{\infty} \Big( \frac{\int_{B_{t}(x)} |y|^{\sigma_1}v(y)^q \,dy}{t^{n-\beta\gamma}} \Big)^{\frac{1}{\gamma-1}} \,\frac{dt}{t} \leq Cu(z) \,\text{ for all }\, z \in B_{\delta}(x). $$
Hence, $u(x) \leq \epsilon + Cu(z)$ for $z \in B_{\delta}(x)$, which implies $$u(x)^{r_0} \leq C_{1}\epsilon^{r_0} + C_{2}u(z)^{r_0}.$$ Integrating this inequality over the ball $B_{\delta}(x)$ implies
\begin{equation}\label{estimate 3}
u(x)^{r_0} \leq C_{1}\epsilon^{r_0} + \frac{C_{2}}{|B_{\delta}(x)|}\int_{B_{\delta}(x)} u(z)^{r_0} \,dz. 
\end{equation} 
Since $u \in L^{r_0}(\mathbb{R}^n)$, 
$$\frac{1}{|B_{\delta}(x)|}\int_{B_{\delta}(x)} u(z)^{r_0} \,dz \longrightarrow 0 \,\text{ as }\, |x| \longrightarrow \infty,$$ 
and thus the right-hand side of \eqref{estimate 3} tends to zero as $|x| \longrightarrow \infty$ and $\epsilon \longrightarrow 0$. Hence, $u(x) \longrightarrow 0$ as $|x|\longrightarrow \infty$. This completes the proof of the theorem.
\end{proof}

\begin{corollary}\label{corollary}
There holds $$\int_{\mathbb{R}^n} |y|^{\sigma_1}v(y)^q \,dy < \infty.$$
\end{corollary}

\begin{proof}
There are two cases to consider. 

(i) First, assume $n-\beta\gamma \leq p(\frac{n-\beta\gamma}{\gamma-1}) - (\beta\gamma +\sigma_2)$. Indeed, since $q \geq p$ and $\sigma_1 \leq \sigma_2 $ in  $(-\beta\gamma,0]$, the non-subcritical condition \eqref{non-subcritical} implies $q \geq \frac{(\gamma-1)(n+\beta\gamma + 2\sigma_1)}{n-\beta\gamma}$. Now choose an appropriate $\varepsilon > 0$ so that $\varepsilon \in (\beta\gamma + 2\sigma_1,\beta\gamma + \sigma_1)$ and let $\ell = \frac{n+\varepsilon}{n+\beta\gamma+2\sigma_1}$ and $\ell' = \frac{n+\varepsilon}{\varepsilon - 2\sigma_1 - \beta\gamma}$. Therefore, $\frac{1}{\ell} + \frac{1}{\ell'} = 1$ with $\ell q > \frac{n(\gamma-1)}{n-\beta\gamma}$ and $\ell' > \frac{n}{-\sigma_1}$. Thus, H\"{o}lder's inequality and Theorem \ref{integrability theorem} imply
\begin{align*}
\int_{B_{1}(0)^{C}} |y|^{\sigma_1}v(y)^q \,dy \leq {} & \Big( \int_{B_{1}(0)^{C}} |y|^{\sigma_1 \frac{\ell}{\ell-1}} \,dy \Big)^{\frac{\ell-1}{\ell}} \Big(\int_{B_{1}(0)^{C}} v(y)^{\ell q} \,dy\Big)^{1/\ell} \\
\leq {} & C\Big( \int_{R}^{\infty} t^{n+\sigma_1 \ell'} \,\frac{dt}{t}\Big)^{\frac{1}{\ell'}}  \Big(\int_{B_{1}(0)^{C}} v(y)^{\ell q} \,dy\Big)^{1/\ell} \\
\leq {} & C\|v\|_{\ell q}^{q} < \infty.
\end{align*} 

(ii) Now assume $n-\beta\gamma > p(\frac{n-\beta\gamma}{\gamma-1}) - (\beta\gamma +\sigma_2)$. For some $\varepsilon \in (0,-\sigma_1)$, set $\ell = \frac{n}{n+\sigma_1 + \varepsilon}$ and $\ell' = \frac{n}{-\sigma_1 - \epsilon}$ so that $\frac{1}{\ell} + \frac{1}{\ell'} = 1$. Indeed, $pq > (\gamma-1)^2$ and the non-subcritical condition \eqref{non-subcritical} imply
\begin{align*}
\frac{(n+\sigma_1)(\gamma-1) + q(n+\sigma_2)}{q(\gamma-1 +p)} = {} & \frac{(n+\sigma_1)(\gamma-1)}{q(\gamma-1+p)} + \frac{n+\sigma_2}{\gamma-1+p}\\
< {} & \frac{n+\sigma_1}{\gamma - 1 + q} + \frac{n+\sigma_2}{\gamma -1 + p} 
\leq \frac{n-\beta\gamma}{\gamma-1}.
\end{align*}
It follows that
$$(n+\sigma_1)(\gamma-1)\frac{1}{q} < \frac{(n-\beta\gamma)(\gamma-1+p)}{\gamma-1} - (n+\sigma_2) = p(\frac{n-\beta\gamma}{\gamma-1}) - (\beta\gamma + \sigma_2)$$
and thus $\frac{n+\sigma_1}{n}\frac{1}{q} < \frac{p(\frac{n-\beta\gamma}{\gamma-1}) - (\beta\gamma + \sigma_2)}{n(\gamma-1)}.$
Then for $\varepsilon$ sufficiently small, 
$$lq > \frac{n(\gamma-1)}{p(\frac{n-\beta\gamma}{\gamma-1}) - (\beta\gamma + \sigma_2)} = \max \Bigg\lbrace \frac{n(\gamma-1)}{n-\beta\gamma}, \frac{n(\gamma-1)} {p(\frac{n-\beta\gamma}{\gamma-1}) - (\beta\gamma +\sigma_2) }\Bigg\rbrace.$$
Hence, H\"{o}lder's inequality and Theorem \ref{integrability theorem} imply
$$\int_{B_{1}(0)^C} |y|^{\sigma_1}v(y)^q \,dy \leq C\|v\|_{lq}^{q} < \infty. $$
Moreover, $\int_{B_{1}(0)} |y|^{\sigma_1}v(y)^q \,dy \leq C\|v\|_{\infty}^{q} < \infty$ since $\sigma_1 \in (-n,0]$. Hence, these calculations imply that $|y|^{\sigma_1}v^q \in L^{1}(\mathbb{R}^n)$.
\end{proof}

\begin{remark}
Of course if $\sigma_i = 0$, the preceding corollary states that $v$ belongs to $L^{q}(\mathbb{R}^n)$. Thus, since $\gamma-1 + q > q$ and the non-subcritical condition implies that $\gamma-1 + p > n(\gamma-1)(n-\beta\gamma)^{-1}$, we see from Corollary \ref{corollary} and Theorem \ref{integrability theorem} that $u,v$ are also finite-energy solutions i.e. $(u,v) \in L^{p+\gamma-1}(\mathbb{R}^n)\times L^{q+\gamma-1}(\mathbb{R}^n)$.
\end{remark}

The next result is a key step in establishing the fast decay rates of integrable solutions. Although we state the theorem here, we delay its proof until the final section in order to better illustrate the main ideas in our proof of Theorem \ref{fast theorem}. We remark that our need for this key result is due to the variable coefficients in the integral system. In contrast, for the constant coefficient case and with the help of the method of moving planes in integral form, the integrable solutions are indeed radially symmetric and the arguments for establishing the decay estimates become far simpler. However, for variable coefficients, the solutions may no longer have any radial symmetry. 

Here $\varphi \in C_{0}^{\infty}(B_{1}(0)\backslash B_{1/2}(0))$ is a cut-off function where $0\leq \varphi(x) \leq 1$ for $1/2\leq |x| \leq 1$ and $\varphi(x) = 1$ for $5/8 \leq |x| \leq 7/8$. Then there holds the following.

\begin{theorem}\label{boundedness property}
Let $(u,v)$ be a positive integrable solution of the integral system \eqref{Wolff} and set $\varphi_{r}(x) \doteq \varphi(\frac{x}{r})$ for any $r > 0$ and $$g(x) = v(x)|x|^{\frac{n+\sigma_1}{q}}\varphi_{r}(x).$$ Then there exists a positive constant $C$ independent of $r$ such that
\begin{equation}\label{cut-off}
g(x) \leq C \,\text{ for all }\, x.
\end{equation}
\end{theorem}

\section{Fast decay rates of positive solutions}\label{section fast decay}
Throughout this section, $u,v$ are understood to be positive integrable solutions of system \eqref{Wolff} unless further specified.

\subsection{Fast decay rate for $u(x)$}
\begin{proposition}\label{bounded below}
For suitably large $|x|$, there exists a positive constant $c$ such that
$$ u(x),v(x) \geq c|x|^{-\frac{n-\beta\gamma}{\gamma-1}}. $$
\end{proposition}

\begin{proof}
For large $|x|$, it is clear that
\begin{align*}
u(x) \geq {} & c\int_{1 + |x|}^{\infty} \Big( \frac{\int_{B_{1}(0)} |y|^{\sigma_1}v(y)^q \,dy}{t^{n-\beta\gamma}} \Big)^{\frac{1}{\gamma-1}} \,\frac{dt}{t} \notag \\
\geq {} & c\int_{1+|x|}^{\infty} t^{\frac{\beta\gamma - n}{\gamma - 1}} \,\frac{dt}{t} \geq c|x|^{-\frac{n-\beta\gamma}{\gamma-1}}.
\end{align*}
The lower bound for $v(x)$ follows similarly and this completes the proof.
\end{proof}

\begin{proposition}\label{decay u}
There holds $u(x) \simeq |x|^{-\frac{n-\beta\gamma}{\gamma-1}}$.
\end{proposition}

\begin{proof}

In view of Proposition \ref{bounded below}, it only remains to show that there exists a positive constant $C$ such that 
\begin{equation}
u(x) \leq C|x|^{-\frac{n-\beta\gamma}{\gamma-1}} \,\text{ for suitably large }\, |x|.
\end{equation}

We consider two cases: (i) Let $t \leq |x|/2$. Then $y \in B_{t}(x)$ implies that $|x|/2 \leq |y| \leq 3|x|/2$, and by virtue of Theorem \ref{boundedness property}, 
$$|y|^{\sigma_1}v(y)^{q} \leq C|y|^{\sigma_1}(|y|^{-\frac{n+\sigma_1}{q}})^{q} \leq C|y|^{-n} \leq C|x|^{-n}.$$
Hence, there holds
\begin{equation*}
\int_{0}^{|x|/2} \Big( \frac{\int_{B_{t}(x)} |y|^{\sigma_1}v(y)^q \,dy}{t^{n-\beta\gamma}} \Big)^{\frac{1}{\gamma-1}} \,\frac{dt}{t} 
\leq C|x|^{-\frac{n}{\gamma-1}}\int_{0}^{|x|/2} t^{\frac{\beta\gamma}{\gamma-1}} \,\frac{dt}{t} \leq C|x|^{-\frac{n-\beta\gamma}{\gamma-1}}.
\end{equation*}
(ii) Suppose $t > |x|/2$. According to Corollary \ref{corollary}, $|x|^{\sigma_1}v(x)^q\in L^{1}(\mathbb{R}^n)$, so
\begin{equation*} 
\int_{|x|/2}^{\infty} \Big( \frac{\int_{B_{t}(x)} |y|^{\sigma_1}v(y)^q \,dy}{t^{n-\beta\gamma}} \Big)^{\frac{1}{\gamma-1}} \,\frac{dt}{t} \leq C \int_{|x|/2}^{\infty} t^{\frac{\beta\gamma-n}{\gamma - 1}} \,\frac{dt}{t} \leq C|x|^{-\frac{n-\beta\gamma}{\gamma-1}}.
\end{equation*}
By combining the last two estimates, we arrive at
$$ u(x) = c_{1}(x)W_{\beta,\gamma}(|y|^{\sigma_1}v^q)(x) \leq C|x|^{-\frac{n-\beta\gamma}{\gamma-1}} \,\text{ for large }\, |x|.$$
This completes the proof.

\end{proof}

\subsection{Fast decay rates for $v(x)$}

\begin{proposition}\label{decay v1}
If $p(\frac{n-\beta\gamma}{\gamma-1}) - \sigma_2 > n$, then $v(x) \simeq |x|^{-\frac{n-\beta\gamma}{\gamma-1}}$. 
\end{proposition}

\begin{proof}
Consider the splitting
\begin{equation*}
v(x) \leq C\Big( \int_{0}^{|x|/2} + \int_{|x|/2}^{\infty} \Big) \Big( \frac{\int_{B_{t}(x)} |y|^{\sigma_2}u(y)^p \,dy}{t^{n-\beta\gamma}} \Big)^{\frac{1}{\gamma-1}} \,\frac{dt}{t} \doteq C(I_1 + I_2).
\end{equation*}
For large $|x|$ there holds
\begin{equation*}
I_1 = \int_{0}^{|x|/2} \Big( \frac{\int_{B_{t}(x)} |y|^{\sigma_2}u(y)^p \,dy}{t^{n-\beta\gamma}} \Big)^{\frac{1}{\gamma-1}} \,\frac{dt}{t} \leq C|x|^{p\frac{\beta\gamma-n}{(\gamma-1)^2} + \frac{\beta\gamma + \sigma_2}{\gamma-1}} \leq C|x|^{-\frac{n-\beta\gamma}{\gamma-1}},
\end{equation*}
since $p(\frac{n-\beta\gamma}{\gamma-1}) - \sigma_2 > n$ and $u(x) \leq C|x|^{-\frac{n-\beta\gamma}{\gamma-1}}$. 

It remains to estimate $I_2$. Using similar calculations in the proof of Corollary \ref{corollary}, we can show that $|y|^{\sigma_2}u^p \in L^{1}(\mathbb{R}^n)$ in this case. Therefore, 
\begin{equation*}
I_2 = \int_{|x|/2}^{\infty} \Big( \frac{\int_{B_{t}(x)} |y|^{\sigma_2}u(y)^p \,dy}{t^{n-\beta\gamma}} \Big)^{\frac{1}{\gamma-1}} \,\frac{dt}{t} \leq C|x|^{-\frac{n-\beta\gamma}{\gamma-1}}.
\end{equation*}
Hence, these estimates for $I_1$ and $I_2$ together with Proposition \ref{bounded below} complete the proof.
\end{proof}

\begin{proposition}\label{decay v2}
If $p(\frac{n-\beta\gamma}{\gamma-1}) - \sigma_2 = n$, then $v(x) \simeq |x|^{-\frac{n-\beta\gamma}{\gamma-1}}(\ln |x|)^{\frac{1}{\gamma-1}}$. 
\end{proposition}

\begin{proof}
{\it Step 1:} For any $\lambda > 1$ if $t > \lambda|x|$, then $B_{t-|x|}(0) \subset B_{t}(x)$. Then from Proposition \ref{bounded below}, we can find a suitably small $R>0$ such that
\begin{align*}
v(x) \geq {} & c\int_{\lambda|x|}^{\infty} \Big( \frac{\int_{B_{t-|x|}(0)\backslash B_{R}(0)} |y|^{\sigma_2}u(y)^p \,dy}{t^{n-\beta\gamma}}  \Big)^{\frac{1}{\gamma-1}} \,\frac{dt}{t} \\
\geq {} & c\int_{\lambda |x|}^{\infty} \Big( \frac{\int_{R}^{t-|x|} r^{n + \sigma_2 - p(\frac{n-\beta\gamma}{\gamma-1})} \,\frac{dr}{r} }{t^{n-\beta\gamma}}  \Big)^{\frac{1}{\gamma-1}} \,\frac{dt}{t} \\
\geq {} & c\int_{\lambda |x|}^{\infty} \Big( \frac{\int_{1-1/\lambda}^{(1-1/\lambda)t} \,\frac{dr}{r} }{t^{n-\beta\gamma}}  \Big)^{\frac{1}{\gamma-1}} \,\frac{dt}{t}
\geq c \int_{\lambda |x|}^{\infty} \Big( \frac{\ln t}{t^{n-\beta\gamma}} \Big)^{\frac{1}{\gamma-1}} \,\frac{dt}{t}.
\end{align*}
From this we deduce that 
\begin{equation}\label{J0}
\lim_{|x|\longrightarrow\infty} \frac{|x|^{\frac{n-\beta\gamma}{\gamma-1}}}{(\ln |x|)^{\frac{1}{\gamma - 1}}}v(x) \geq c > 0,
\end{equation}
which follows after sending $\lambda \longrightarrow 1$ in the following identity (see (4.1) in \cite{SunLei2012}):
\begin{equation}\label{log property} 
\lim_{|x|\longrightarrow \infty} \frac{|x|^{\frac{n-\beta\gamma}{\gamma-1}}}{(\ln \lambda |x|)^{\frac{1}{\gamma - 1}}}\int_{\lambda |x|}^{\infty} \Big( \frac{\ln t}{t^{n-\beta\gamma}} \Big)^{\frac{1}{\gamma-1}} \,\frac{dt}{t} = \frac{\gamma-1}{n-\beta\gamma}\lambda^{-\frac{n-\beta\gamma}{\gamma-1}} \,~\, (\lambda > 0).
\end{equation}
\medskip

\noindent{\it Step 2:} We estimate the terms $J_1$ and $J_2$ where
\begin{align*}
v(x) \leq C\Big( \int_{0}^{\lambda |x|} + \int_{\lambda |x|}^{\infty} \Big) \Big( \frac{\int_{B_{t}(x)} |y|^{\sigma_2}u(y)^p \,dy}{t^{n-\beta\gamma}} \Big)^{\frac{1}{\gamma-1}} \,\frac{dt}{t} \doteq C(J_1 + J_2)
\end{align*}
with $\lambda \in (1/2,1)$ and $|x|$ is large. Then for $0\leq t \leq \lambda |x|$ and $y\in B_{t}(x)$, Proposition \ref{decay u} implies that $u(y) \leq C|y|^{-\frac{n-\beta\gamma}{\gamma-1}} \leq C|x|^{-\frac{n-\beta\gamma}{\gamma-1}}$. Therefore,
\begin{equation*}
J_{1} \leq C|x|^{-\frac{p(n-\beta\gamma)}{(\gamma-1)^2} + \frac{\sigma_2}{\gamma-1}}\int_{0}^{\lambda |x|} t^{\frac{\beta\gamma}{\gamma-1}} \,\frac{dt}{t} \leq C|x|^{-\frac{p(n-\beta\gamma)}{(\gamma-1)^2} + \frac{\sigma_2 + \beta\gamma}{\gamma-1}} \leq C|x|^{-\frac{n-\beta\gamma}{\gamma-1}} ,
\end{equation*}
since $p\frac{n-\beta\gamma}{\gamma-1} - \sigma_2 = n$ implies
$$-\frac{p(n-\beta\gamma)}{(\gamma-1)^2} + \frac{\sigma_2 + \beta\gamma}{\gamma-1} = -\frac{n-\beta\gamma}{\gamma-1}. $$
Hence,
\begin{equation}\label{J1}
\lim_{|x|\longrightarrow\infty} |x|^{\frac{n-\beta\gamma}{\gamma-1}}(\ln |x|)^{-\frac{1}{\gamma-1}}J_1 = 0.
\end{equation}
In view of $B_{t}(x) \subset B_{|x|+t}(0)$ and Jensen's inequality, we can write
\begin{align*}
J_{2} \leq {} & C\int_{\lambda |x|}^{\infty} \Big( \frac{\int_{B_{1}(0)} |y|^{\sigma_2} u(y)^p \,dy + \int_{B_{|x| + t}(0)\backslash B_{1}(0)} |y|^{\sigma_2} u(y)^p \,dy }{t^{n-\beta\gamma}} \Big)^{\frac{1}{\gamma-1}} \,\frac{dt}{t} \\
\leq {} & C\int_{\lambda |x|}^{\infty} \Big( \frac{\int_{B_{1}(0)} |y|^{\sigma_2} u(y)^p \,dy}{t^{n-\beta\gamma}} \Big)^{\frac{1}{\gamma-1}} + \Big( \frac{\int_{B_{|x| + t}(0)\backslash B_{1}(0)} |y|^{\sigma_2} u(y)^p \,dy}{t^{n-\beta\gamma}} \Big)^{\frac{1}{\gamma-1}} \,\frac{dt}{t} \\
\doteq {} & C(J_3 + J_4).
\end{align*}
Since $\int_{B_{1}(0)} |y|^{\sigma_2}u(y)^p \,dy \leq C$,
\begin{align*}
J_{3} \leq C\int_{\lambda |x|}^{\infty} \Big( \frac{\int_{B_{1}(0)} |y|^{\sigma_2} u(y)^p \,dy}{t^{n-\beta\gamma}} \Big)^{\frac{1}{\gamma-1}} \,\frac{dt}{t} \leq C\int_{\lambda |x|}^{\infty} t^{\frac{\beta\gamma - n}{\gamma-1}} \,\frac{dt}{t} \leq C|x|^{-\frac{n-\beta\gamma}{\gamma-1}}.
\end{align*}
Likewise, Proposition \ref{decay u} implies that
\begin{align*}
J_{4} \leq {} & C\int_{\lambda |x|}^{\infty} \Big( \frac{\int_{B_{t+|x|}(0)\backslash B_{1}(0)} |y|^{\sigma_2 - p(\frac{n-\beta\gamma}{\gamma-1})} \,dy }{t^{n-\beta\gamma}} \Big)^{\frac{1}{\gamma-1}} \,\frac{dt}{t} \\
\leq {} & C \int_{\lambda |x|}^{\infty} \Big( \frac{\int_{1}^{t+|x|} r^{\sigma_2 - p(\frac{n-\beta\gamma}{\gamma-1}) + n} \,\frac{dr}{r} }{t^{n-\beta\gamma}} \Big)^{\frac{1}{\gamma-1}} \,\frac{dt}{t}
\leq C\int_{\lambda |x|}^{\infty} \Big( \frac{\ln t}{t^{n-\beta\gamma}} \Big)^{\frac{1}{\gamma-1}} \,\frac{dt}{t}.
\end{align*}
In view of \eqref{log property}, sending $\lambda \longrightarrow 1$ yields 
\begin{equation}\label{J2}
\lim_{|x|\longrightarrow \infty} \frac{|x|^{\frac{n-\beta\gamma}{\gamma-1}}}{(\ln |x|)^{\frac{1}{\gamma-1}}}J_2 \leq C.
\end{equation}
Hence, \eqref{J0},\eqref{J1} and \eqref{J2} imply $v(x) \simeq |x|^{-\frac{n-\beta\gamma}{\gamma-1}}(\ln |x|)^{\frac{1}{\gamma-1}}$.
\end{proof}

\begin{proposition}\label{decay v3}
If $p(\frac{n-\beta\gamma}{\gamma-1}) - \sigma_2 < n$, then $v(x) \simeq |x|^{-\frac{p(\frac{n-\beta\gamma}{\gamma-1}) - (\beta\gamma + \sigma_2)}{\gamma-1}}$. 
\end{proposition}

\begin{proof}
Fix a suitable $R>0$ and let
$$\Omega_1 \doteq [|x|-R,|x|+R] \,\text{ and }\, \Omega_2 \doteq [1 - R/|x|,1 + R/|x|]$$ 
and consider the splitting
\begin{align*}
v(x) \leq {} & C\Big(\int_{\Omega_1} + \int_{\Omega_{1}^{c}} \Big) \Big( \frac{\int_{B_{t}(x)} |y|^{\sigma_2} u(y)^p \,dy}{t^{n-\beta\gamma}} \Big)^{\frac{1}{\gamma-1}} \,\frac{dt}{t} \\
\doteq {} & C(K_1 + K_2).
\end{align*}
\medskip 

\noindent{\it Step 1:} We claim that $$\lim_{|x|\longrightarrow\infty} |x|^{\frac{p(\frac{n-\beta\gamma}{\gamma-1}) - (\beta\gamma + \sigma_2)}{\gamma-1}}K_1 = 0.$$
Since $B_{t}(x) \subset B_{2t + R}(0)$ whenever $t \in \Omega_1$, Proposition \ref{decay u} implies that
\begin{align*}
K_1 \leq {} & C\int_{\Omega_1} \Big( \frac{\int_{0}^{2t+R} r^{n + \sigma_2 - p(\frac{n-\beta\gamma}{\gamma-1})} \,\frac{dr}{r}}{t^{n-\beta\gamma}} \Big)^{\frac{1}{\gamma-1}} \,\frac{dt}{t} \\
\leq {} & C|x|^{-\frac{p(\frac{n-\beta\gamma}{\gamma-1}) - (\beta\gamma + \sigma_2) + 1}{\gamma-1}}.
\end{align*}
Hence,  
$$\lim_{|x|\longrightarrow\infty} |x|^{\frac{p(\frac{n-\beta\gamma}{\gamma-1}) - (\beta\gamma + \sigma_2)}{\gamma-1}} K_1 = 0,$$
and this proves the claim.
\medskip

\noindent{\it Step 2:} We show $$\lim_{|x|\longrightarrow\infty} |x|^{\frac{p(\frac{n-\beta\gamma}{\gamma-1}) - (\beta\gamma + \sigma_2)}{\gamma-1}}K_2 = C.$$
As before, Proposition \ref{decay u} implies that for large $|x|$,
\begin{align*}
K_2 \leq {} & C\int_{\Omega_{1}^{c}} \Big( \frac{\int_{B_{t}(x)} |y|^{-p(\frac{n-\beta\gamma}{\gamma-1}) + \sigma_2} \,dy}{t^{n-\beta\gamma}} \Big)^{\frac{1}{\gamma-1}} \,\frac{dt}{t} \\
\leq {} & C|x|^{-\frac{p(\frac{n-\beta\gamma}{\gamma-1}) - (\beta\gamma + \sigma_2)}{\gamma-1}}\int_{\Omega_{2}^{c}} \Big( \frac{\int_{B_{s}(e)} |z|^{-p(\frac{n-\beta\gamma}{\gamma-1})+\sigma_2} \,dz}{s^{n-\beta\gamma}} \Big)^{\frac{1}{\gamma-1}} \,\frac{ds}{s} \\
\leq {} & C|x|^{-\frac{p(\frac{n-\beta\gamma}{\gamma-1}) - (\beta\gamma + \sigma_2)}{\gamma-1}}.
\end{align*}
Here, we have used the change of variables $z = \frac{y}{|x|}$ and $s = \frac{t}{|x|}$ and the assumption that
\begin{equation}\label{change of variables}
\int_{0}^{\infty} \Big( \frac{\int_{B_{s}(e)} |z|^{-p(\frac{n-\beta\gamma}{\gamma-1})+\sigma_2} \,dz}{s^{n-\beta\gamma}} \Big)^{\frac{1}{\gamma-1}} \,\frac{ds}{s} < \infty,
\end{equation}
where $e = x/|x|$ is a unit vector. Likewise, using the lower bound estimate of $u$ in Proposition \ref{bounded below}, we can apply similar arguments to show $$K_2 \geq C|x|^{-\frac{p(\frac{n-\beta\gamma}{\gamma-1}) - (\beta\gamma + \sigma_2)}{\gamma-1}},$$
and we deduce that
$$\lim_{|x|\longrightarrow\infty} |x|^{\frac{p(\frac{n-\beta\gamma}{\gamma-1}) - (\beta\gamma + \sigma_2)}{\gamma-1}}K_2 = C.$$
Therefore, it only remains to prove assertion \eqref{change of variables}. We do so by considering the splitting
$$\Big(\int_{0}^{1/2} + \int_{1/2}^{\infty} \Big) \Big( \frac{\int_{B_{s}(e)} |z|^{-p(\frac{n-\beta\gamma}{\gamma-1})+\sigma_2} \,dz}{s^{n-\beta\gamma}} \Big)^{\frac{1}{\gamma-1}} \,\frac{ds}{s} \doteq K_3 + K_4.$$
Since $|z| \in [1/2,3/2]$ whenever $z \in B_{s}(e)$, there holds
\begin{align*}
K_{3} \leq C\int_{0}^{1/2} \Big( \frac{|B_{s}(e)|}{s^{n-\beta\gamma}} \Big)^{\frac{1}{\gamma-1}}\,\frac{ds}{s} \leq C\int_{0}^{1/2} s^{\frac{\beta\gamma}{\gamma-1}}\,\frac{ds}{s} < \infty.
\end{align*}
On the other hand, we can certainly find a suitably large $c>0$ so that
\begin{align*}
K_4 \leq {} & C\int_{1/2}^{\infty}\Big( \frac{\int_{B_{c s}(0)} |z|^{-p(\frac{n-\beta\gamma}{\gamma-1})+\sigma_2} \,dz}{s^{n-\beta\gamma}} \Big)^{\frac{1}{\gamma-1}} \,\frac{ds}{s} \\
\leq {} & C\int_{1/2}^{\infty}\Big( \frac{\int_{0}^{c s} r^{n-p(\frac{n-\beta\gamma}{\gamma-1})+\sigma_2} \,\frac{dr}{r}}{s^{n-\beta\gamma}} \Big)^{\frac{1}{\gamma-1}} \,\frac{ds}{s} \\
\leq {} & C\int_{1/2}^{\infty} s^{-\frac{p(\frac{n-\beta\gamma}{\gamma-1}) - (\beta\gamma + \sigma_2)}{\gamma-1}} \,\frac{ds}{s} < \infty,
\end{align*}
and this completes the proof.
\end{proof}

\begin{proof}[\bf Proof of Theorem \ref{fast theorem}]
If $u,v$ are positive integrable solutions, then Theorem \ref{integrability theorem} and Propositions \ref{decay u}--\ref{decay v3} imply $u,v$ are bounded and decay with the fast rates as $|x|\longrightarrow\infty$. Conversely, assume $u,v$ are bounded and decay with the fast rates as $|x|\longrightarrow\infty$. If $u(x)$ decays with the rate $|x|^{-\frac{n-\beta\gamma}{\gamma-1}}$, then 
\begin{align*}
\int_{\mathbb{R}^n} u(x)^{r_0} \,dx \leq {} & \int_{B_{1}(0)} u(x)^{r_0} \,dx + \int_{\mathbb{R}^n \backslash B_{1}(0)} u(x)^{r_0} \,dx \\
\leq {} & C_{1} + C_{2}\int_{1}^{\infty} t^{n-(\frac{n-\beta\gamma}{\gamma-1})r_0} \,\frac{dt}{t} < \infty,
\end{align*}
since the non-subcritical condition implies $(n-\beta\gamma)r_0 > n(\gamma-1)$. Likewise, if $v(x)$ decays with the rate $|x|^{-\frac{n-\beta\gamma}{\gamma-1}}$, we can show $v \in L^{s_0}(\mathbb{R}^n)$. If $v(x)$ decays with the rate $|x|^{-\frac{n-\beta\gamma}{\gamma-1}}(\ln |x|)^{\frac{1}{\gamma-1}}$, then we can find a suitably large $R>0$ and small $\varepsilon > 0$ such that 
$$(\ln |x|)^{\frac{s_0}{\gamma-1}} \leq C|x|^{\varepsilon} \,\text{ for }\, |x| > R.$$ 
This implies 
$$ \int_{\mathbb{R}^n} v(x)^{s_0} \,dx \leq C_1 + C_{2}\int_{R}^{\infty} t^{n-(\frac{n-\beta\gamma}{\gamma-1})s_0 + \varepsilon} \,\frac{dt}{t} < \infty,$$
since $n-(\frac{n-\beta\gamma}{\gamma-1})s_0 + \varepsilon < 0$ provided $\varepsilon$ is sufficiently small. Now suppose $v(x)$ decays with the rate 
$$|x|^{-\frac{(p\frac{n-\beta\gamma}{\gamma-1}) - (\beta\gamma + \sigma_2)}{\gamma-1}}. $$ 
Since $q_0 < \frac{n-\beta\gamma}{\gamma-1}$, we obtain $p(\frac{n-\beta\gamma}{\gamma-1}) - (\beta\gamma + \sigma_2) > pq_0 - (\beta\gamma + \sigma_2) = p_0 (\gamma-1)$. From this we deduce that 
$$n - \Big(p(\frac{n-\beta\gamma}{\gamma-1}) - (\beta\gamma + \sigma_2)\Big)\frac{s_0}{\gamma-1} < 0$$ 
and thus
$$ \int_{\mathbb{R}^n} v(x)^{s_0} \,dx \leq C_1 + C_{2}\int_{1}^{\infty} t^{n- (p(\frac{n-\beta\gamma}{\gamma-1}) - (\beta\gamma + \sigma_2))\frac{s_0}{\gamma -1}} \,\frac{dt}{t} < \infty.$$
Hence, in any case, we conclude that $u,v$ are integrable solutions. This completes the proof of the theorem.
\end{proof}

\section{Proof of Corollary \ref{cor2}}\label{proof of corollary}

Let $(u,v)$ be a positive solution of system \eqref{quasilinear}. If $u,v$ are either the integrable solutions or are bounded and decay with the fast rates as $|x|\longrightarrow\infty$, then clearly
\begin{equation}\label{inf}
\inf_{\mathbb{R}^n} u = \inf_{\mathbb{R}^n} v = 0.
\end{equation} 
Thus, the potential estimate of Corollary 4.13 from \cite{KM94} ensures positive constants $C_1$  and $C_2$ such that
\begin{align*}
C_1 W_{1,\gamma}(c_{1}(y)|y|^{\sigma_1}v^q)(x) \leq u(x) \leq C_2 W_{1,\gamma}(c_{1}(y)|y|^{\sigma_1}v^q)(x),\\ \medskip
C_1 W_{1,\gamma}(c_{2}(y)|y|^{\sigma_2}u^p)(x) \leq v(x) \leq C_2 W_{1,\gamma}(c_{2}(y)|y|^{\sigma_2}u^p)(x).
\end{align*}
Since $c_{1}(x)$ and $c_{2}(x)$ are double bounded, we can then take $k_{1}(x)$ and $k_{2}(x)$ to be the double bounded functions
$$ k_{1}(x) = \frac{u(x)}{ W_{1,\gamma}(|y|^{\sigma_1}v^q)(x) } \,\text{ and }
   k_{2}(x) = \frac{v(x)}{W_{1,\gamma}(|y|^{\sigma_2}u^p)(x)} $$
so that $u,v$ satisfies the integral system
\begin{equation}\label{k system}
  \left\{\begin{array}{l}
	u(x) = k_{1}(x)W_{1,\gamma}(|y|^{\sigma_2}v^q)(x), \\
	v(x) = k_{2}(x)W_{1,\gamma}(|y|^{\sigma_2}u^p)(x). 
  \end{array}
\right.
\end{equation}
Therefore, the desired result follows immediately from Theorem \ref{fast theorem}.
\qed

\section{Proof of Theorem \ref{boundedness property}}\label{proof of theorem 3}
On the contrary, assume \eqref{cut-off} does not hold. Then there exists an increasing sequence $ r_j \longrightarrow \infty$ as $j\longrightarrow \infty$ such that if $x_{r_j}$ denotes the maximum point of $g$ in $B_{r_j}(0)\backslash B_{\frac{r_j}{2}}(0)$, then
\begin{equation*}
\lim_{j\longrightarrow\infty} g(x_{r_j}) = \infty.
\end{equation*}
Thus,
\begin{equation}\label{contradict}
v(x_{r_j}) = \frac{g(x_{r_j})}{|x_{r_j}|^{\frac{n+\sigma_1}{q}}\varphi_{r_j}(x_{r_j})} \geq \frac{c}{|x_{r_j}|^{\frac{n+\sigma_1}{q}}}.
\end{equation}
As was done in \cite{LL12}, there holds $\varphi_{r_j}(x_{r_j}) > \delta$ for some small $\delta \in (0,1)$ independent of $r_j$. Therefore, we can find a small $s>0$ such that
\begin{equation*}
\varphi_{r_j}(y) > \delta/2 \,\text{ for }\, y \in B_{s|x_{r_j}|}(x_{r_j}).
\end{equation*}
Since $g(y) \leq g(x_{r_j})$, we have
\begin{equation*}
v(y) \leq C\frac{v(x_{r_j})}{\varphi_{r_j}(y)} \leq C\delta v(x_{r_j}).
\end{equation*}
By denoting the maximum point of $u$ in $B_{s_{1}|x_{r_j}|}(x_{r_j})$ by $\overline{x}_{r_j}$ for $s_1 \in (0,s)$, which ensures $\overline{x}_{r_j}$ lies in the interior of $ B_{s|x_{r_j}|}(x_{r_j})$, we get 
$$u(y) \leq u(\overline{x}_{r_j}) \,\text{ for all }\, y \in B_{s_{2}|x_{r_j}|}(\overline{x}_{r_j}) \subset B_{s|x_{r_j}|}(x_{r_j})$$ with $s_2 \in (0,s-s_1)$. 
\medskip

\noindent{\it Step 1:} We claim that for small $\varepsilon \in (0,1)$, there is $\varepsilon_1 \geq 0$ such that for $|x_{r_j}|$ large,
\begin{equation}\label{step1}
u(\overline{x}_{r_j}) \leq \varepsilon v(x_{r_j})^{1+\varepsilon_1} + C|x_{r_j}|^{-\frac{n-\beta\gamma}{\gamma-1}}.
\end{equation}
To prove this assertion, consider the splitting of the first equation in \eqref{Wolff},
\begin{align*}
u(\overline{x}_{r_j}) \leq {} & C\Big( \int_{0}^{s_{2}|x_{r_j}|} + \int_{s_{2}|x_{r_j}|}^{\infty} \Big) \Big( \frac{\int_{B_{t}(\overline{x}_{r_j})} |y|^{\sigma_1}v(y)^q \,dy }{t^{n-\beta\gamma}} \Big)^{\frac{1}{\gamma-1}} \,\frac{dt}{t} \\
\doteq {} & C(L_1 + L_2).
\end{align*}
The second term can be directly bounded since $|y|^{\sigma_1}v^q \in L^{1}(\mathbb{R}^n)$ implies
\begin{align*}
L_{2} \leq C\int_{s_{2}|x_{r_j}|}^{\infty} t^{\frac{\beta\gamma - n}{\gamma-1}} \,\frac{dt}{t} \leq C|x_{r_j}|^{-\frac{n-\beta\gamma}{\gamma-1}}.
\end{align*}
To estimate $L_1$, let $\rho \in (0,s_{2}|x_{r_j}|)$ and consider
\begin{align*}
L_{1} \leq {} & Cv(x_{r_j})^{1+\varepsilon_1} \Big( \int_{0}^{\rho} + \int_{\rho}^{s_{2}|x_{r_j}|} \Big) \Big( \frac{\int_{B_{t}(\overline{x}_{r_j})} |y|^{\sigma_1}v(y)^{q - (1+\varepsilon_1)(\gamma-1)} \,dy }{t^{n-\beta\gamma}} \Big)^{\frac{1}{\gamma-1}} \,\frac{dt}{t} \\
\doteq {} & Cv(x_{r_j})^{1+\varepsilon_1}(L_{11}+L_{12}).
\end{align*}
Then for $|x_{r_j}|$ sufficiently large, the ground state properties of $v$ imply
\begin{equation*}
L_{11} \leq \varepsilon^{\frac{q-(\varepsilon_1 + 1)(\gamma-1)}{\gamma-1}}\int_{0}^{\rho} t^{\frac{\beta\gamma + \sigma_1}{\gamma-1}} \,\frac{dt}{t} \leq \varepsilon/2C.
\end{equation*} 
If $|x_{r_j}|$ and $\rho$ are large, H\"{o}lder's inequality and Corollary \ref{corollary} imply
\begin{align*}
\int_{B_{t}(\overline{x}_{r_j})}  {} & |y|^{\sigma_{1}}v(y)^{q - (1+\varepsilon_1)(\gamma-1)} \,dy \\
= {} & \int_{B_{t}(\overline{x}_{r_j})} |y|^{\sigma_{1} \frac{q-(1+\varepsilon_1)(\gamma-1)}{q}}v(y)^{q - (1+\varepsilon_1)(\gamma-1)} |y|^{ \frac{\sigma_{1} (1 + \varepsilon_1)(\gamma-1)}{q}} \,dy \\
\leq {} & \Big( \int_{B_{t}(\overline{x}_{r_j})} |y|^{\sigma_1}v(y)^q \,dy \Big)^{1 - (1+\varepsilon_1)(\gamma-1)/q} \Big( \int_{B_{t}(\overline{x}_{r_j})} |y|^{\sigma_1} \,dy \Big)^{\frac{(1+\varepsilon_1)(\gamma-1)}{q}} \\
\leq {} & \||y|^{\sigma_1}v^q \|_{1}^{1 - (1+\varepsilon_1)(\gamma-1)/q} t^{\frac{(n+\sigma_1)(1+\varepsilon_1)(\gamma-1)}{q}} 
\leq C t^{\frac{(n+\sigma_1)(1+\varepsilon_1)(\gamma-1)}{q}}.
\end{align*}
Therefore,
\begin{align*}
L_{12} \leq {} & C\int_{\rho}^{s_{2}|x_{r_j}|} \Big( \frac{\int_{B_{t}(\overline{x}_{r_j})} |y|^{\sigma_{1}}v(y)^{q - (1+\varepsilon_1)(\gamma-1)} \,dy }{t^{n-\beta\gamma}} \Big)^{\frac{1}{\gamma-1}} \,\frac{dt}{t} \\
\leq {} & C\int_{\rho}^{s_{2}|x_{r_j}|} t^{-\frac{n-\beta\gamma}{\gamma-1} + \frac{(n+\sigma_1)(1+\varepsilon_1)}{q}  } \,\frac{dt}{t}
\leq \varepsilon/2C
\end{align*}
since $-\frac{n-\beta\gamma}{\gamma-1} + \frac{(n+\sigma_1)(1+\varepsilon_1)}{q} < 0$ provided we choose $\varepsilon_1$ in $[0,\frac{n-\beta\gamma}{\gamma-1}\frac{q}{n+\sigma_1} - 1)$.
Hence, we arrive at $L_{1} \leq \varepsilon v(x_{r_j})^{1+\varepsilon_1}$ and combining this with the estimate for $L_2$ yields \eqref{step1}.
\medskip

\noindent{Step 2:} For large $|x_{r_j}|$, there exists $\varepsilon_2 \geq 0$ such that 
\begin{equation}\label{step2}
v(x_{r_j}) \leq u(\overline{x}_{r_j})^{\frac{1}{1+\varepsilon_1}} + C|x_{r_j}|^{-\frac{n-\beta\gamma - \varepsilon_2}{\gamma-1}}
\end{equation}
Consider the splitting of the second integral equation
\begin{align*}
v(x_{r_j}) \leq {} & C\Big( \int_{0}^{s_{1}|x_{r_j}|} + \int_{s_{1}|x_{r_j}|}^{\infty} \Big) \Big( \frac{\int_{B_{t}(x_{r_j})} |y|^{\sigma_2} u(y)^p \,dy }{t^{n-\beta\gamma}} \Big)^{\frac{1}{\gamma-1}} \,\frac{dt}{t} \\
\doteq {} & C(L_3 + L_4).
\end{align*}

Actually, if $p(\frac{n-\beta\gamma}{\gamma-1}) - \sigma_2 > n$, then we can choose $\varepsilon_1 = 0$ in \eqref{step1}. Likewise, since $q \geq p$, there holds $q(\frac{n-\beta\gamma}{\gamma-1}) - \sigma_1 > n$ and we can mimic the proof in Step 1 on the second integral equation to get estimate \eqref{step2} with $\varepsilon_1 = \varepsilon_2 = 0$. Therefore, we arrive at the estimates
\begin{equation*}
  \left\{\begin{array}{l}
    u(\overline{x}_{r_j}) \leq \varepsilon v(x_{r_j}) + C|x_{r_j}|^{-\frac{n-\beta\gamma}{\gamma-1}}, \\
    v(x_{r_j}) \leq \varepsilon u(\overline{x}_{r_j}) + C|x_{r_j}|^{-\frac{n-\beta\gamma}{\gamma-1}},
  \end{array}
\right.
\end{equation*}
and it follows that $$v(x_{r_j}) \leq  C|x_{r_j}|^{-\frac{n-\beta\gamma}{\gamma-1}}.$$
This is a contradiction with \eqref{contradict} and this completes the proof of the theorem for this case. Hence, in view of this, we restrict our attention to the case where $p(\frac{n-\beta\gamma}{\gamma-1}) - \sigma_2 \leq n$. 
\medskip

\noindent (i) First, let $p(\frac{n-\beta\gamma}{\gamma-1}) - \sigma_2 = n$ and we estimate $L_3$ and $L_4$ from above. By definition of $\overline{x}_{r_j}$, we have
\begin{align*}
L_{3} \leq {} &  u(\overline{x}_{r_j})^{\frac{1}{1+\varepsilon_1}} \Big(\int_{0}^{\rho} + \int_{\rho}^{s_{1}|x_{r_j}|} \Big) \Big( \frac{\int_{B_{t}(x_{r_j})} |y|^{\sigma_2} u(y)^{p - \frac{\gamma-1}{1+\varepsilon_1}} \,dy }{t^{n-\beta\gamma}} \Big)^{\frac{1}{\gamma-1}} \,\frac{dt}{t} \\
\doteq {} & u(\overline{x}_{r_j})^{\frac{1}{1+\varepsilon_1}} (L_{31} + L_{32}).
\end{align*}
By virtue of the boundedness and decaying property of $u$, we obtain for $|x_{r_j}|$ large that
$$ L_{31} \leq \varepsilon^{\frac{p(1+\varepsilon_1) - (\gamma-1)}{(1+\varepsilon_1)(\gamma-1)}} \int_{0}^{\rho} t^{\frac{\beta\gamma}{\gamma-1}} \,\frac{dt}{t} \leq \varepsilon/2C.$$
For any given $\varepsilon_1 > 0$ and because $p(\frac{n-\beta\gamma}{\gamma-1}) - \sigma_2 = n$, we can show
\begin{equation}\label{tedious1}
\frac{n(\gamma-1)}{n-\beta\gamma} < \frac{n(p - \frac{\gamma-1}{1+\varepsilon_1})}{\beta\gamma + \sigma_2}.
\end{equation}
Hence, we may choose $\ell$ in the interval
$ \Big( \frac{n(\gamma-1)}{n-\beta\gamma}, \frac{n(p - \frac{\gamma-1}{1+\varepsilon_1})}{\beta\gamma + \sigma_1} \Big)$
so that H\"{o}lder's inequality and Theorem \ref{integrability theorem} yield
\begin{align*}
\int_{B_{t}(x_{r_j})} |y|^{\sigma_2}u(y)^{p-\frac{\gamma-1}{1+\varepsilon_1}}\,dy 
\leq {} & Ct^{\sigma_1}\|u\|_{\ell}^{p-\frac{\gamma-1}{1+\varepsilon_1}} |B_{t}(x_{r_j})|^{1 - \frac{p(1+\varepsilon_1) - (\gamma-1)}{(1+\varepsilon_1)\ell}} \\
\leq {} & Ct^{n + \sigma_1 - a},
\end{align*}
where $a \doteq \frac{n}{\ell}\frac{p(1+\varepsilon_1) - (\gamma-1)}{1+\varepsilon_1}$ and $\rho \leq t \leq s_{1}|x_{r_j}|$.
Therefore, $\beta\gamma + \sigma_1 - a < 0$ and we get
$$ L_{32} \leq C\int_{\rho}^{s_{1}|x_{r_j}|} t^{\frac{ \beta\gamma + \sigma_1 - a }{\gamma-1}} \,\frac{dt}{t} \leq \varepsilon/2C. $$
Now, it is simple to show that
\begin{equation}\label{tedious2}
-\frac{\sigma_2(\gamma-1)}{n-\beta\gamma} < \frac{np}{\beta\gamma + \sigma_2} -p.
\end{equation}
Thus, we can choose $\varepsilon_{2}'$ in $(-\frac{\sigma_2(\gamma-1)}{n-\beta\gamma},\frac{np}{\beta\gamma + \sigma_2} - p)$ so that $n+ \sigma_2 (1+ p/\varepsilon_{2}') > 0$.
If $t \geq s_{1}|x_{r_j}|$ and $y \in B_{t}(x_{r_j})$, then $|y| \leq t + |x_{r_j}| \leq (1+1/s_{1})t$ and thus $B_{t}(x_{r_j}) \subset B_{(1+1/s_{1})t}(0) \equiv B_{ct}(0)$. Thus, H\"{o}lder's inequality and Theorem \ref{integrability theorem} imply
\begin{align*}
\int_{B_{t}(x_{r_j})} |y|^{\sigma_{2}} u(y)^p \,dy 
\leq {} & \int_{B_{ct}(0)} |y|^{\sigma_{2}} u(y)^p \,dy \\
\leq {} & \|u\|_{p+\varepsilon_{2}'}^{p}\Big( \int_{B_{ct}(0)} |y|^{\sigma_2 (1 + p/\varepsilon_{2}')} \,dy \Big)^{\frac{\varepsilon_{2}'}{p + \varepsilon_{2}'}} \\
\leq {} & C\Big( \int_{0}^{ct} r^{n + \sigma_2 (1 + p/\varepsilon_{2}')} \,\frac{dr}{r} \Big)^{\frac{\varepsilon_{2}'}{p + \varepsilon_{2}'}} \\
\leq {} & C t^{\frac{n\varepsilon_{2}'}{p + \varepsilon_{2}'} +\sigma_2 } = C t^{n+\sigma_2 - \frac{np}{p+\varepsilon_{2}'}}.
\end{align*}
Hence,
\begin{equation*} 
L_{4} \leq C \int_{s_{1}|x_{r_j}|}^{\infty} t^{\frac{\beta\gamma + \sigma_2 - np/(p+\varepsilon_{2}') }{\gamma-1}} \,\frac{dt}{t} \leq C|x_{r_j}|^{-\frac{n -\beta\gamma - \varepsilon_2}{\gamma-1}},
\end{equation*}
where $\varepsilon_2 = \frac{n\varepsilon_{2}'}{p+ \varepsilon_{2}'} + \sigma_2$. Combining these estimates for $L_3$ and $L_4$ give us \eqref{step2}.
\medskip

\noindent (ii) Let $p(\frac{n-\beta\gamma}{\gamma-1}) - \sigma_2 < n$. We estimate $L_3$ and $L_4$ in this case. The estimate for $L_3$ follows just as in the previous case provided we choose $\varepsilon_1$ to be in the interval
$$ \Big(\frac{(n+\sigma_2)(\gamma-1) - p(n-\beta\gamma)}{p(n-\beta\gamma)-(\beta\gamma+\sigma_2)(\gamma-1)}, \frac{q(n-\beta\gamma)}{(n+\sigma_1)(\gamma-1)}-1 \Big). $$
This is possible since the non-subcritical condition yields $q_0 < \frac{n-\beta\gamma}{\gamma-1}$ which, after some direct calculations, implies that
$$\frac{(n+\sigma_2)(\gamma-1) - p(n-\beta\gamma)}{p(n-\beta\gamma)-(\beta\gamma+\sigma_2)(\gamma-1)} < \frac{q(n-\beta\gamma)}{(n+\sigma_1)(\gamma-1)}-1.$$
Meanwhile, this choice for $\varepsilon_1$ implies \eqref{tedious1} and the same arguments apply thereafter. Let us now estimate $L_4$ for this case. By the non-subcritical condition and $p(\frac{n-\beta\gamma}{\gamma-1}) - \sigma_2 < n$, we can show
$$ p \in \Big( \frac{(\gamma-1)(\beta\gamma + \sigma_2)}{n-\beta\gamma}, \frac{(n+\sigma_2)(\gamma-1)}{n-\beta\gamma} \Big)$$
and thus $\frac{n(\gamma-1)}{n-\beta\gamma} < \frac{np}{\beta\gamma + \sigma_2}$. Since $n > \beta\gamma$ we have that $\frac{np}{\beta\gamma +\sigma_2} > \frac{np}{n+\sigma_2}$. Therefore, we can choose $\varepsilon_{2}' >0$ in the interval
\begin{equation}\label{tedious3}
\Big(\frac{np}{n + \sigma_2} - \frac{n(\gamma-1)}{n-\beta\gamma},\, \frac{np}{\beta\gamma + \sigma_2} - \frac{n(\gamma-1)}{n-\beta\gamma} \Big)
\end{equation} 
so that for $t \geq s_{1}|x_{r_j}|$, H\"{o}lder's inequality and Theorem \ref{integrability theorem} imply
\begin{align*}
\int_{B_{t}(x_{r_j})} {} &  |y|^{\sigma_{2}} u(y)^p \,dy 
\leq \int_{B_{ct}(0)} |y|^{\sigma_{2}} u(y)^p \,dy \\
\leq {} & \|u\|_{\frac{n(\gamma-1)}{n-\beta\gamma} + \varepsilon_{2}'}^{p}\Big( \int_{B_{ct}(0)} |y|^{\sigma_2 \frac{n(\gamma-1)+ \varepsilon_{2}'(n-\beta\gamma)}{n(\gamma-1)+ \varepsilon_{2}'(n-\beta\gamma) - p(n-\beta\gamma)}} \,dy \Big)^{1-\frac{p(n-\beta\gamma)}{n(\gamma-1)+ \varepsilon_{2}'(n-\beta\gamma)}} \\
\leq {} & C\Big( \int_{0}^{ct} r^{n + \sigma_2 \frac{n(\gamma-1)+ \varepsilon_{2}'(n-\beta\gamma)}{n(\gamma-1)+ \varepsilon_{2}'(n-\beta\gamma) - p(n-\beta\gamma)}} \,\frac{dr}{r} \Big)^{1-\frac{p(n-\beta\gamma)}{n(\gamma-1)+ \varepsilon_{2}'(n-\beta\gamma)}} \\
\leq {} &  C t^{n+\sigma_2 - \frac{np(n-\beta\gamma)}{n(\gamma-1)+ \varepsilon_{2}'(n-\beta\gamma)}}.
\end{align*}
Hence,
\begin{equation*}
L_{4} \leq C\int_{s_{1}|x_{r_j}|}^{\infty} t^{\frac{\beta\gamma +\sigma_2 - \frac{np(n-\beta\gamma)}{n(\gamma-1)+ \varepsilon_{2}'(n-\beta\gamma)}}{\gamma-1} } \,\frac{dt}{t} \leq C|x_{r_j}|^{-\frac{n-\beta\gamma - \varepsilon_2 }{\gamma-1}},
\end{equation*}
where $$\varepsilon_{2} = n + \sigma_2 - \frac{np(n-\beta\gamma)}{n(\gamma-1) + \varepsilon_{2}' (n-\beta\gamma)}$$
and $\varepsilon_{2} > 0$ due to \eqref{tedious3}. Thus, combining the estimates for $L_3$ and $L_4$ leads to \eqref{step2}.
\medskip

\noindent{\it Step 3:} (iii) Let $p(\frac{n-\beta\gamma}{\gamma-1}) - \sigma_2 = n$. Choose $\varepsilon_1 \in (0,\frac{q(n-\beta\gamma)}{(n+\sigma_1)(\gamma-1)}-1)$ so that estimate \eqref{step1} holds. Applying estimate \eqref{step2} to estimate \eqref{step1} yields
\begin{align}\label{absorb}
u(\overline{x}_{r_j}) \leq {} & \varepsilon \Big( u(\overline{x}_{r_j})^{\frac{1}{1+\varepsilon_1}} + C|x_{r_j}|^{-\frac{n-\beta\gamma-\varepsilon_2}{\gamma-1}} \Big)^{1+\varepsilon_1} + C|x_{r_j}|^{-\frac{n-\beta\gamma}{\gamma-1}} \notag \\
\leq {} & C\varepsilon u(\overline{x}_{r_j}) + C|x_{r_j}|^{-\frac{(1+\varepsilon_1)(n-\beta\gamma-\varepsilon_2)}{\gamma-1}} + C|x_{r_j}|^{-\frac{n-\beta\gamma}{\gamma-1}} \notag \\
\leq {} & C\varepsilon u(\overline{x}_{r_j}) + C|x_{r_j}|^{-\frac{n-\beta\gamma}{\gamma-1}}
\end{align}
since $$\frac{(1+\varepsilon_1)(n-\beta\gamma-\varepsilon_2)}{\gamma-1} > \frac{n-\beta\gamma}{\gamma-1}$$
provided we choose $\varepsilon_{2}'$ in $(-\frac{\sigma_2 (\gamma-1)}{n-\beta\gamma}, \,\min\lbrace \frac{(n+\sigma_2)\varepsilon_1 (\gamma-1) - \sigma_2 (1+\varepsilon_1)p}{n+\beta\gamma\varepsilon_1 + (1+\varepsilon_1)\sigma_2}, \frac{np}{\beta\gamma + \sigma_2}-p \rbrace)$. Note that this is possible due to \eqref{tedious2} and the fact that 
\begin{equation*}
-\frac{\sigma_2(\gamma-1)}{n-\beta\gamma} < \frac{(n+\sigma_2)(\gamma-1)\varepsilon_1 - \sigma_2 (1+\varepsilon_1)p}{n+\beta\gamma\varepsilon_1 + (1+\varepsilon_1)\sigma_2}.
\end{equation*}
After absorbing the first term on the right-hand side of \eqref{absorb} by the left-hand side, we get the estimate
$$ u(\overline{x}_{r_j}) \leq C|x_{r_j}|^{-\frac{n-\beta\gamma}{\gamma-1}}.$$
Applying this estimate to estimate \eqref{step2} yields
$$ v(x_{r_j}) \leq C|x_{r_j}|^{-\frac{n-\beta\gamma}{(\gamma-1)(1+\varepsilon_1)}},$$
but this contradicts with \eqref{contradict} in view of $\varepsilon_1 < \frac{q(n-\beta\gamma)}{(n+\sigma_1)(\gamma-1)} - 1$.
\medskip

\noindent (iv) If $p(\frac{n-\beta\gamma}{\gamma-1}) - \sigma_2 < n$, we can adopt the same arguments as in part (iii) to arrive at a contradiction provided we choose $\varepsilon_1$ to be in the interval
$$ \Big(\frac{(n+\sigma_2)(\gamma-1) - p(n-\beta\gamma)}{p(n-\beta\gamma)-(\beta\gamma+\sigma_2)(\gamma-1)}, \frac{q(n-\beta\gamma)}{(n+\sigma_1)(\gamma-1)}-1 \Big) $$
and we choose a positive and suitably small $\varepsilon_{2}' < \frac{(n+\sigma_2)\varepsilon_{1}(\gamma-1) - \sigma_{2}(1+\varepsilon_1)p}{n + \beta\gamma \varepsilon_1 + (1+\varepsilon_1)\sigma_2}$ which also belongs to the interval \eqref{tedious3}. This completes the proof of the theorem. \qed

\medskip

\small
\noindent{\bf Acknowledgment:} 
This work was completed during a visiting appointment at the University of Oklahoma, and the author would like to express his sincere appreciation to the university and the Department of Mathematics for their hospitality.


\end{document}